\documentclass[11pt,a4paper]{amsart}
\pdfoutput=1
\usepackage{times}
\usepackage{amssymb}

\usepackage{ifthen}
\usepackage{cite}
\usepackage{graphicx}
\usepackage{xr}

\makeatletter
\newcommand{\ga}{\alpha}
\newcommand{\gb}{\beta}
\newcommand{\gd}{\delta}
\newcommand{\gep}{\epsilon}

\renewcommand{\gg}{\gamma}

\newcommand{\gl}{\lambda}

\newcommand{\gs}{\sigma}

\newcommand{\gD}{\Delta}

\newcommand{\gO}{\Omega}

\newcommand{\cA}{\mathcal{A}}
\newcommand{\cB}{\mathcal{B}}

\newcommand{\cD}{\mathcal{D}}

\newcommand{\cF}{\mathcal{F}}
\newcommand{\cG}{\mathcal{G}}

\newcommand{\cJ}{\mathcal{J}}

\newcommand{\cL}{\mathcal{L}}

\newcommand{\1}{1}

\newcommand{\C}{\mathbb{C}}

\newcommand{\N}{\mathbb{N}}

\newcommand{\R}{\mathbb{R}}

\newcommand{\p}{\partial}

\newtheorem{theorem}{Theorem}[section]

\newtheorem{lemma}[theorem]{Lemma}
\newtheorem{corollary}[theorem]{Corollary}

\theoremstyle{remark}

\newtheorem{remark}[theorem]{Remark}

\newcommand{\erf}{\mathop{\operator@font erf}\nolimits}
\newcommand{\erfc}{\mathop{\operator@font erfc}\nolimits}
\newcommand{\sign}{\mathop{\operator@font sign}\nolimits}

\newif\if@golden  \@goldentrue
\newcommand{\f@ctor}{1}
\newlength{\aiv@width}  
\newlength{\aiv@height} 
\newlength{\tmp@width}  
\newlength{\tmp@height} 
\if@twocolumn\if@golden
  \else\fi
\else\if@golden\ifcase\@ptsize\relax\or
\or\fi
  \else\ifcase\@ptsize\relax\or
\or\fi\fi
\if@twocolumn\if@golden\ifcase\@ptsize\relax\renewcommand{\f@ctor}{53}
  \or\renewcommand{\f@ctor}{46}\or\renewcommand{\f@ctor}{43}\fi
  \else\ifcase\@ptsize\relax\renewcommand{\f@ctor}{51}\or
  \renewcommand{\f@ctor}{45}\or\renewcommand{\f@ctor}{42}\fi\fi
\else\if@golden\ifcase\@ptsize\relax \renewcommand{\f@ctor}{46}
  \or\renewcommand{\f@ctor}{43}\or\renewcommand{\f@ctor}{43}\fi
  \else\ifcase\@ptsize\relax\renewcommand{\f@ctor}{43}
  \or\renewcommand{\f@ctor}{40}\or\renewcommand{\f@ctor}{40}\fi\fi\fi
\multiply\textheight by \f@ctor
\let\comp\circ

\newcommand{\op}{\ensuremath^\circ}

\newcommand{\cGo}{\ensuremath\cG\op}

\newcommand{\sgl}{\ensuremath\sqrt{2\gl}}
\newcommand{\sn}{\ensuremath\{1,2,\dotsc,n\}}
\newcommand{\da}{\ensuremath\downarrow}
\newcommand{\ua}{\ensuremath\uparrow}

\newcommand{\onel}[1][k]{\ensuremath1_{l_{#1}}}
\newcommand{\cond}{\ensuremath\,\big|\,}
\newcommand{\Co}{\ensuremath C_0(\cG)}

\newcommand{\Cii}{\ensuremath C_0^2(\cG)}

\newcommand{\kn}{\ensuremath k=1,\dotsc,n}
\newcommand{\Ieqref}[1]{\textup{\tagform@{I.\ref{I_#1}}}}
\newcommand{\Iref}[1]{I.\ref{I_#1}}
\newcommand{\IIeqref}[1]{\textup{\tagform@{II.\ref{II_#1}}}}

\newcommand{\pW}{\ensuremath p^w}
\newcommand{\rW}{\ensuremath r^w}
\newcommand{\RW}{\ensuremath R^w}
\newcommand{\SW}{\ensuremath S^w}
\newcommand{\TW}{\ensuremath H^w}

\newcommand{\cFW}{\ensuremath \cF^w}
\newcommand{\We}{\ensuremath W^e}
\newcommand{\ReW}{\ensuremath R^e}

\newcommand{\reW}{\ensuremath r^e}
\newcommand{\peW}{\ensuremath p^e}
\newcommand{\TeW}{\ensuremath H^e}
\newcommand{\SeW}{\ensuremath S^e}
\newcommand{\seW}{\ensuremath s^e}
\newcommand{\Ws}{\ensuremath W^s}
\newcommand{\RsW}{\ensuremath R^s}

\newcommand{\rsW}{\ensuremath r^s}
\newcommand{\psW}{\ensuremath p^s}
\newcommand{\TsW}{\ensuremath H^s}
\newcommand{\SsW}{\ensuremath S^s}

\newcommand{\cFsW}{\ensuremath \cF^s}
\newcommand{\Wg}{\ensuremath W^g}
\newcommand{\Rg}{\ensuremath R^g}

\newcommand{\rg}{\ensuremath r^g}
\newcommand{\pg}{\ensuremath p^g}

\newcommand{\Sg}{\ensuremath S^g}

\newcommand{\hgO}{\ensuremath{\hat\gO}}

\newcommand{\hP}{\ensuremath{\hat P}}
\newcommand{\hE}{\ensuremath{\hat E}}
\newcommand{\hR}{\ensuremath{\hat R}}
\newcommand{\hcA}{\ensuremath{\hat \cA}}

\newcommand{\hX}{\ensuremath{\hat X}}

\newlength{\BCs@ze}
\newlength{\BCsh@ft}
\ifcase\@ptsize\relax
\or
\or
\fi
\DeclareFixedFont\MT{OMS}{cmsy}{m}{n}{\BCs@ze}    % standard 'Computer Modern' symbols
\newcommand{\BigCart}{\ensuremath\mathop{\raisebox{\BCsh@ft}{{\MT\char"02}}}}

\def\pdftitle{\@gobble}
\let\setdif\setminus

\makeatother

\numberwithin{equation}{section}
\allowdisplaybreaks[2]

\date{December 6, 2010}

\title[Brownian Motions on Metric Graphs II]{%
Brownian Motions on Metric Graphs II\\
{\small Construction of Brownian Motions on Single Vertex Graphs}
}

\pdftitle{Brownian Motions on Metric Graphs II -
          Construction of Brownian Motions on Single Vertex Graphs}

\author[V.~Kostrykin]{Vadim Kostrykin}
\address{Vadim Kostrykin\newline
Institut f\"ur Mathematik\newline
Johannes Gutenberg--Universit\"at\newline
D--55099 Mainz, Germany}
\email{kostrykin@mathematik.uni-mainz.de}

\author[J.~Potthoff]{J\"urgen Potthoff}
\address{J\"urgen Potthoff\newline
Institut f\"ur Mathematik\newline
Universit\"at Mannheim\newline
D--68131 Mann\-heim, Germany}
\email{potthoff@math.uni-mannheim.de}

\author[R.~Schrader]{Robert Schrader}
\address{Robert Schrader\newline
Institut f\"{u}r Theoretische Physik\newline
Freie Universit\"{a}t Berlin, Arnimallee~14\newline
D--14195 Berlin, Germany}
\email{schrader@physik.fu-berlin.de}

\copyrightinfo{2010}{V.~Kostrykin, J.~Potthoff, R.~Schrader}

\begin{document}
\begin{abstract}
Pathwise constructions of Brownian motions which satisfy all possible boundary
conditions at the vertex of single vertex graphs are given.
\end{abstract}

\maketitle
\thispagestyle{empty}

\section{Introduction and Preliminaries} \label{sect1}
This is the second in a series of three articles about the construction and the
basic properties of Brownian motions on metric graphs. In the first of these
articles~\cite{BMMG1}, Brownian motions on metric graphs have been defined, their
Feller property has been shown, and their generators have been determined, i.e., the
analogue of Feller's theorem for metric graphs has been proved. In the present
article, we construct all possible Brownian motions on single vertex graphs (cf.\
below). In the third article of this series~\cite{BMMG3}, this construction will be
extended to general metric graphs. In a companion article~\cite{BMMG0}, which serves
more as a background for this series and which can also be read as an introduction
to the topic, we revisit the classical cases~\cite{Fe54, Fe54a, Fe57, ItMc63,
ItMc74, Kn81} of Brownian motions on bounded intervals and on the semi-line $\R_+$.

For a general introduction to the subject of this article we refer the interested
reader to~\cite{BMMG1}, henceforth quoted as ``article~I''. We shall refer to
equations, definitions, theorems etc.\ from article~I by placing an ``I'' in front.
For example, ``formula~\Ieqref{eq_2_4}'' refers to formula~(\ref{I_eq_2_4}) in
article~I, while ``definition~\Iref{def_3_1}'' points to definition~\ref{I_def_3_1} in
article~I. Unless otherwise mentioned, we continue to use the notation and the
conventions set up in that article.

In the present article we solely consider a metric graph $\cG$ consisting of a single
vertex $v$, and $n$, $n<+\infty$,  external edges labeled $l_k$, $k\in\sn$, each of
which is metrically isomorphic to the interval $[0,+\infty)$, the vertex corresponding
to $0$. It will be convenient to denote the local coordinate of a point $\xi$ in
$\cGo=\cG\setdif\{v\}$ by $(k,x)$, $k\in\sn$, $x\in(0,+\infty)$, instead of
$(l_k,x)$ as we did in paper~I.

The main ideas for the construction of Brownian motions with boundary
conditions at the vertex compatible with Feller's theorem~\ref{I_thm_5_3} are those
which can be found in the work by It\^o and McKean~\cite{ItMc63, ItMc74} for the
case of the semi-line $\R_+$, cf.\ also~\cite[Chapter~6]{Kn81}: The reflecting
Brownian motion in the case of $\R_+$ is replaced by a Walsh process~{\cite{Wa78}}
(cf.\ also, e.g., \cite{BaPi89}) on the single vertex graph, and then the killing and
slowing down of this process on the scale of its local time at the vertex is used to
construct processes implementing the various forms of the Wentzell boundary
condition. On the other hand, we provide a number of arguments which --- at least on
a technical level --- are rather different from those found in the standard
literature. For example, whenever possible, we use arguments based on Dynkin's formula
to derive the domain of the generator (i.e., the boundary conditions). This approach
appears to be much simpler and more intuitive than the one with standard
arguments~\cite{ItMc63, ItMc74, Kn81} for the semi-line $\R_+$, which is based on
the rather tricky calculation of heat kernels with the help of L\'evy's theorem.
Moreover, for the case of killing, instead of using the standard first passage time
formula for the hitting time of the vertex we use a first passage time formula for
the life time of the process. In the opinion of the authors this leads to much
simpler computations of the transition kernels than those in~\cite{ItMc63, ItMc74,
Kn81} for a Brownian motion on $\R_+$.

The article is organized as follows. In several subsections of the present section
we set up some notation and discuss some preparatory results. In section~\ref{sect2}
we recall the construction of a Walsh process on a metric graph. In
section~\ref{sect3} we construct a Walsh process on the single vertex graph with an
elastic boundary condition at the vertex, while in section~\ref{sect4} we construct
a Walsh process with a sticky boundary condition at the vertex. The most general
Brownian motion on the single vertex graph is obtained in section~\ref{sect5} by
combining these two constructions. In all three cases we also derive explicit expressions
of the analogues of the quantum mechanical scattering matrix on single vertex
graphs.

\subsection{Feller's Theorem and Boundary Conditions}   \label{ssect1.1}
For ease of later reference, in this subsection we state Feller's theorem,
theorem~\Iref{thm_5_3}, as it reads for a single vertex graph with $n$ edges.

$\Co$ denotes the space of continuous functions on $\cG$ which
vanish at infinity, and $\Cii$ the subspace of $\Co$ consisting of those
functions $f\in\Co$ which are twice continuously differentiable on $\cGo$, such that
$f'$ vanishes at infinity and $f''$ belongs to $\Co$. Moreover, for a function $f$
which is continuous on $\cGo$, and such that for $k\in\sn$, $f(\xi)$ has a limit when
$\xi\in\cGo$ approaches the vertex $v$ along any edge $l_k$, we set
\begin{equation*}
    f(v_k) = \lim_{\xi\to v,\,\xi\in l_k} f(\xi).
\end{equation*}
It is not hard to see that for all $f\in\Cii$, $k\in\sn$, the limits $f'(v_k)$ of
the derivatives exist. Theorem~\Iref{thm_5_3} states that the generator $A$ (on
$\Co$) of a Brownian motion on $\cG$ is given by $1/2$ times the Laplacean with
domain $\cD(A)\subset\Cii$, and that there exist non-negative constants $a$, $c$,
$a\ne 1$, and a vector $b=(b_k,\,\kn)\in [0,1]^n$, with
\begin{subequations}    \label{eq1.1}
\begin{equation}    \label{eq1.1a}
    a + c + \sum_{k=1}^n b_k = 1,
\end{equation}
and $\cD(A)$ consists exactly of those $f$ for which
\begin{equation}    \label{eq1.1b}
    a f(v) + \frac{c}{2}\,f''(v) = \sum_{k=1}^n b_k f'(v_k)
\end{equation}
\end{subequations}
holds true.

\subsection{Standard Brownian Motion on the Real Line}   \label{ssect1.2}
The construction of Brownian motions on a single vertex graph with infinitesimal generator
whose domain consists of $f$'s which satisfy the
boundary conditions~\eqref{eq1.1} is quite similar to the construction carried out
for the half-line in~\cite{ItMc63}, \cite{ItMc74}, \cite{Kn81}. This in turn is
based on the properties of a standard Brownian motion on the real line, cf., e.g.,
\cite{Fe71, Hi80, IkWa89, KaSh91, ReYo91, Wi79}, and the works cited above. For the
convenience of the reader, and for later reference, we collect the pertinent
notions, tools and results here.

Let $(Q_x,\,x\in\R)$ denote a family of probability measures on a measurable space
$(\gO',\cA')$, and let $B=(B_t,\,t\in\R_+)$ denote a standard Brownian motion
defined on $(\gO',\cA')$ with $Q_x(B_0=x)=1$, $x\in\R$. It will be convenient
to assume throughout that $B$ has only continuous paths. Whenever it is notationally
convenient, we shall also write $B(t)$ for $B_t$, $t\ge 0$. Furthermore, we may
suppose that there is a shift operator $\theta: \R_+\times\gO \to\gO$, such that for
all $s$, $t\ge 0$, $B_s\comp\theta_t=B_{s+t}$.

We shall always understand the Brownian family $(B, (Q_x,\,x\in\R))$ to be equipped
with a filtration $\cJ=(\cJ_t,\,t\ge 0)$ which is right continuous and complete for
the family $(Q_x,\,x\in\R)$. (For example, $\cJ$ could be chosen as the usual
augmentation of the natural filtration of $B$ (e.g., \cite[Sect.~2.7]{KaSh91} or
\cite[Sect.'s~I.4, III.2]{ReYo91}).)

For any $A\subset\R$, we denote by $H^B_A$ the hitting time of $A$ by $B$,
\begin{equation}    \label{eq1.2}
    H^B_A = \inf\{t>0,\,B_t\in A\},
\end{equation}
and we note that for all $A$ belonging to the Borel $\gs$--algebra $\cB(\R)$ of
$\R$, $H^B_A$ is a stopping time with respect to $\cJ$ (e.g.,
\cite[Theorem~III.2.17]{ReYo91}). In the case where $A=\{x\}$, $x\in\R$, we also
simply write $H^B_x$ for $H^B_{\{x\}}$. We shall also denote these stopping times by
$H^B(A)$ and $H^B(x)$, respectively, whenever it is typographically more convenient.
The following particular cases deserve special attention. Let $x\in\R$. Then we have
(e.g., \cite[Sect.~1.7]{ItMc74}, \cite[Sect.~2.6]{KaSh91}, \cite[Sect.'s~II.3,
III.3]{ReYo91})
\begin{equation}    \label{eq1.3}
    Q_0(H^B_x\in dt) = Q_x(H^B_0\in dt) = \frac{|x|}{t}\,g(t,x)\,dt,\qquad t>0,
\end{equation}
where $g$ is the Gau{\ss}-kernel
\begin{equation}    \label{eq1.4}
    g(t,x) = \frac{1}{\sqrt{2\pi t}}\,e^{-x^2/2t},\qquad t>0,\,x\in\R,
\end{equation}
and
\begin{equation}    \label{eq1.5}
    E^Q_0\bigl(e^{-\gl H^B_x}\bigr)
             = E^Q_x\bigl(e^{-\gl H^B_0}\bigr) = e^{-\sgl |x|},\qquad \gl>0.
\end{equation}
Moreover, for $a<x<b$ the law of $H^B_{\{a,b\}}$ under $Q_x$ is well-known (e.g.,
\cite[Problem~6, Sect.~1.7]{ItMc74}), and its expectation is given by
\begin{equation}    \label{eq1.6}
    E^Q_x\bigl(H^B_{\{a,b\}}\bigr) = (x-a)(b-x).
\end{equation}

Denote by $L^B=(L^B_t,\,t\ge 0)$ the local time of $B$ at zero, where we choose the
normalization as in, e.g., \cite{ReYo91} (and which thus differs by a factor $2$  from
the one used in, e.g., \cite{IkWa89, KaSh91}): for $x\in\R$, $P_x$--a.s.\
\begin{equation}    \label{eq1.7}
    L^B_t = \lim_{\gep\da 0}\,\frac{1}{2\gep}\,
            \gl\bigl(\bigl\{s\le t,\,|B_s|\le \gep\bigr\}\bigr), \qquad t\ge 0,
\end{equation}
and here $\gl$ denotes the Lebesgue measure. Thus, in terms of its $\ga$--potential
(cf.~\cite[Theorem~V.3.13]{BlGe68}) we have
\begin{equation}    \label{eq1.7a}
    u^\ga_{L^B}(x)
        = E_x\Bigl(\int_0^\infty e^{-\ga t}\,dL^B_t\Bigr)
        = \frac{1}{\sqrt{2\ga}}\,e^{-\sqrt{2\ga}|x|},\qquad \ga>0,\,x\in\R,
\end{equation}
which provides an efficient way to compare the various normalizations of the local time
used in the literature. Slightly informally we can write
\begin{equation}    \label{eq1.8}
    L^B_t = \int_0^t \gd_0(B_s)\,ds,
\end{equation}
where $\gd_0$ is the Dirac distribution concentrated at $0$. $L^B$ is adapted to $\cJ$,
and non-decreasing. Moreover, for every $x\in\R$, $P_x$--a.s.\ the paths of $L^B$ are
continuous, and $L^B$ is additive in the sense that
\begin{equation}    \label{eq1.9}
    L^B_{t+s} = L^B_t + L^B_s\comp\theta_t,\qquad s,t\in\R_+.
\end{equation}
Similarly as above, we shall occasionally take the notational freedom to rewrite
$L^B_t$ as $L^B(t)$.

We will need the following well-known result (e.g., \cite[Section~2.2,
Problem~3]{ItMc74}):

\begin{lemma}   \label{lem1.1}
The joint law of $|B_t|$ and $L_t$, $t>0$,  under $Q_0$ is given by
\begin{equation}    \label{eq1.10}
    Q_0\bigl(|B_t|\in dx,\,L^B_t\in dy\bigr)
        = 2\,\frac{x+y}{\sqrt{2\pi t^3}}\,e^{-(x+y)^2/2t}\,dx\,dy,\qquad x,y\ge 0.
\end{equation}
\end{lemma}

Let $K^B = (K^B_r,\,r\ge 0)$ denote the right continuous pseudo-inverse of $L$,
\begin{equation}    \label{eq1.11}
    K^B_r = \inf\bigl\{t\ge 0,\,L^B_t>r\bigr\},\qquad r\ge 0.
\end{equation}
Note that due to the a.s.\ continuity of $L^B$ we have a.s.\ $L^B_{K_r} = r$. In
appendix~\ref{0_app_LT} of~\cite{BMMG0} we prove the following

\begin{lemma}   \label{lem1.2}
For any $r\ge 0$
\begin{equation}    \label{eq1.12}
    Q_0\bigl(K^B_r\in dt\bigr) = \frac{r}{t}\,g(t,r)\,dt,\qquad t>0,
\end{equation}
and
\begin{equation}    \label{eq1.13}
    E^Q_0\bigl(e^{-\gl K^B_r}\bigr) = e^{-\sgl r},\qquad \gl>0
\end{equation}
holds.
\end{lemma}

Moreover, we shall make use of the following lemma, which is similar to results
in Section~6.4 of~\cite{KaSh91}, and which is proved in appendix~\ref{0_app_LT}
of~\cite{BMMG0}.

\begin{lemma}   \label{lem1.3}
Under $Q_0$, $L^B\bigl(H^B_{\{-x,+x\}}\bigr)$, $x>0$, is exponentially distributed
with mean~$x$.
\end{lemma}

\subsection{First Passage Time Formula for Single Vertex Graphs} \label{ssect1.3}
In this subsection we set up some additional notation which will be used throughout
this article. Also we record a special form of the first passage time formula,
equation~\Ieqref{eq_2_6}.

Let $X$ be a Brownian motion on $\cG$ in the sense of definition~\Iref{def_3_1}
defined on a family $\bigl(\gO,\cA, \cF=(\cF_t,\,t\ge 0), (P_\xi,\,\xi\in\cG)\bigr)$
of filtered probability spaces. Let $H_v$ be the hitting time of the vertex $v$.
Note that for all $\xi\in\cG$, $P_\xi(H_v<+\infty)=1$ (cf.\ the discussion in
section~\Iref{sect_3}). For $\gl>0$, set
\begin{equation}    \label{eq1.14}
    e_\gl(\xi) = E_\xi\bigl(\exp(-\gl H_v)\bigr) = e^{-\sgl d(\xi,v)},\qquad \xi\in\cG,
\end{equation}
where $E_\xi(\,\cdot\,)$ denotes expectation with respect to $P_\xi$. The last
equality follows from formula~\eqref{eq1.5}.

Recall that we denote the natural metric on $\cG$ by $d$. We introduce another
symmetric map $d_v$ from $\cG\times\cG$ to $\R_+$ defined by
\begin{equation}    \label{eq1.15}
    d_v(\xi,\eta) = d(\xi,v) + d(v,\eta),\qquad \xi,\,\eta\in\cG,
\end{equation}
which is the ``distance from $\xi$ to $\eta$ via the vertex $v$''. Observe that if
$\xi$, $\eta\in\cG$ do not belong to the same edge, then by the definition of
$d$ the equality $d_v(\xi,\eta)=d(\xi,\eta)$ holds true.

Next we define two heat kernels on $\cG$ by
\begin{align}
    p(t,\xi,\eta)   &= \sum_{k=1}^n \onel(\xi)\, g\bigl(t,d(\xi,\eta)\bigr)\,\onel(\eta),
                    \label{eq1.16}\\
    p_v(t,\xi,\eta) &= \sum_{k=1}^n \onel(\xi)\, g\bigl(t,d_v(\xi,\eta)\bigr)\,\onel(\eta),
                    \label{eq1.17}
\end{align}
with $t>0$, $\xi$, $\eta\in\cG$. $g$ is the Gau{\ss}-kernel defined in
equation~\eqref{eq1.4}. Hence, in local coordinates $\xi=(k,x)$, $\eta=(m,y)$, $x$, $y\ge
0$, $k$, $m\in\sn$, these kernels read
\begin{align}
    p\bigl(t,(k,x),(m,y)\bigr)
        &= \frac{1}{\sqrt{2\pi t}}\,e^{-(x-y)^2/2t}\,\gd_{km}\label{eq1.18}\\
    p_v\bigl(t,(k,x),(m,y)\bigr)
        &= \frac{1}{\sqrt{2\pi t}}\,e^{-(x+y)^2/2t}\,\gd_{km}\label{eq1.19}.
\end{align}
The \emph{Dirichlet heat kernel} $p^D$ on $\cG$ is then given by
\begin{equation}    \label{eq1.20}
    p^D(t,\xi,\eta) = p(t,\xi,\eta) - p_v(t,\xi,\eta),\qquad t>0,\,\xi,\,\eta\in\cG.
\end{equation}
It is the transition density of a strong Markov process with state space
$\cG\op\cup\{\gD\}$ which on every edge of $\cG\op$ is equivalent to a Brownian
motion until the moment of reaching the vertex when it is
killed, and $\gD$ denotes a cemetery state adjoined to $\cG$ as an isolated point.
(Recall from section~\Iref{sect_3} of article~I that this process is \emph{not}
a Brownian motion on $\cG$ in the sense of definition~\Iref{def_3_1}.)

The \emph{Dirichlet resolvent} $R^D = (R^D_\gl,\,\gl>0)$ on $\cG$ is defined by
\begin{equation}    \label{eq1.21}
    R^D_\gl f(\xi)
        = E_\xi\Bigl(\int_0^{H_v} e^{-\gl t} f(X_t)\,dt\Bigr),\qquad
            \gl>0,\,\xi\in\cG,\,f\in B(\cG).
\end{equation}
It is easy to see that $R^D_\gl$ has the following integral kernel on $\cG$
\begin{equation}    \label{eq1.22}
    r^D_\gl(\xi,\eta) = r_\gl(\xi,\eta) - r_{v,\gl}(\xi,\eta),\qquad \xi,\,\eta\in\cG,
\end{equation}
where for $\xi$, $\eta\in\cG$,
\begin{equation}    \label{eq1.22a}
    r_\gl(\xi,\eta) = \sum_{k=1}^n \onel(\xi)\,\frac{e^{-\sgl\,d(\xi,\eta)}}{\sgl}\,\onel(\eta),
\end{equation}
and
\begin{equation}    \label{eq1.22b}
\begin{split}
    r_{v,\gl}(\xi,\eta)
        &= \sum_{k=1}^n \onel(\xi)\, \frac{e^{-\sgl\, d_v(\xi,\eta)}}{\sgl}\,\onel(\eta)\\[1ex]
        &= \sum_{k=1}^n \onel(\xi)\, \frac{1}{\sgl}\,e_\gl(\xi)\, e_\gl(\eta)\,\onel(\eta).
\end{split}
\end{equation}
In particular, $r^D_\gl$ is the Laplace transform of the Dirichlet heat
kernel~\eqref{eq1.20} at $\gl>0$.

\vspace{.5\baselineskip}
Recall the first passage time formula~\Ieqref{eq_2_6}:
\begin{equation*}
    R_\gl f(\xi) = E_\xi\Bigl(\int_0^S e^{-\gl t} f(X_t)\,dt\Bigr)
                    + E_\xi\bigl(e^{-\gl S}\,R_\gl f(X_S)\bigr),
\end{equation*}
where $S$ is any $P_\xi$--a.s.\ finite stopping time relative to $\cJ$.
The choice $S=H_v$ gives the following result.

\begin{lemma}   \label{lem1.4}
Let $X$ be a Brownian motion on $\cG$ with resolvent $R=(R_\gl,\,\gl>0)$. Then
for all $\gl>0$, $\xi\in\cG$, $f\in B(\cG)$,
\begin{equation}    \label{eq1.23}
    R_\gl f(\xi) = R^D_\gl f(\xi) + e_\gl(\xi)\,R_\gl f(v)
\end{equation}
holds true.
\end{lemma}

The following notation will be convenient. For real valued measurable functions
$f$, $g$ on $\cG$, with restrictions $f_k$, $g_k$, $k\in\sn$, to the edges
$l_k$ we set
\begin{equation*}
    (f,g) = \int_{\cG} f(\xi)\,g(\xi)\,d\xi = \sum_{k=1}^n (f_k,g_k),
\end{equation*}
where the integration is with respect to the Lebesgue measure on $\cG$, and
\begin{equation*}
    (f_k,g_k) = \int_0^\infty f_k(x)\,g_k(x)\,dx,
\end{equation*}
whenever the integrals exist.

Assume that $f\in\Co$. Then for $\gl>0$, $R_\gl f$ belongs to the domain of the
ge\-ne\-ra\-tor of $X$, and therefore to $\Cii$ (cf.\ subsection~\ref{ssect1.1}).
It is straightforward to compute the derivative of the right hand side of
formula~\eqref{eq1.23}, and we obtain the

\begin{corollary}   \label{cor1.5}
For every Brownian motion $X$ on $\cG$ with resolvent $R=(R_\gl,\,\gl>0)$,
and all $f\in\Co$,
\begin{equation}    \label{eq1.24}
    (R_\gl f)'(v_k) = 2(e_{\gl,k},f_k) - \sgl R_\gl f(v),\qquad k\in\sn,
\end{equation}
holds true.
\end{corollary}

\subsection{\boldmath The Case $b=0$}     \label{ssect1.4}
The case, where all parameters $b_k$, $\kn$, in equation~\eqref{eq1.1} vanish, is
trivial in the sense that the associated Brownian motion can be constructed by a
stochastic process living only on the edge where it started, and therefore it is
just a classical Brownian motion on $\R_+$ in the sense of~\cite[Section~6.1]{Kn81}.
This case is also discussed briefly in~\cite{Kn81}, but for the sake of completeness
we include it here in somewhat more detail than in~\cite{Kn81}.

Consider a standard Brownian motion on $\R$ as before, and without loss of
generality assume in addition that the underlying sample space is large enough such
that all constant paths in $\R$ can be realized as paths of the Brownian motion.
Construct from  the Brownian motion a new process by stopping it when it reaches the
origin of $\R$, and then kill it after an exponential holding time (independent of
the Brownian motion) with rate $\gb\ge 0$. We shall only consider starting points
$x\in\R_+$. If $\gb=0$, then the process is simply a Brownian motion with absorption
at the origin. For example, it follows from Theorem~10.1 and Theorem~10.2
in~\cite{Dy65a} that for every $\gb\ge 0$ this process is a strong Markov process,
and obviously it has the path properties which make it a Brownian motion on $\R_+$
in the sense of~\cite[Section~6.1]{Kn81}. Thus, if $\xi\in\cG$, $\xi\in l_k$, $\kn$,
then we just have to map this process with the isomorphisms between the edges $l_k$,
$\kn$, and the interval $[0,+\infty)$ into $\cG$ to obtain a Brownian motion on
$\cG$ with start in $\xi$, such that it is stopped when reaching the vertex, and
then is killed there after an exponential holding time with rate $\gb\ge 0$.

Let $U^0=(U^0_t,\,t\ge 0)$ denote the semigroup associated with this process. It is
obvious that for $f\in \Co$ we get $U^0_t f(v) = \exp(-\gb t) f(v)$, $t\ge 0$. Thus
for the corresponding resolvent $R^0=(R^0_\gl,\,\gl>0)$, and $f\in\Co$ one finds
\begin{equation}    \label{eq1.25}
    \gl R^0_\gl f(v) - f(v) + \gb R^0_\gl f(v) = 0,\qquad \gl>0.
\end{equation}
Let $A^0$ be the generator of this process, and recall from lemma~\Iref{lem_5_2},
that for all $f\in\cD(A^0)$, $A^0 f(\xi) = 1/2\, f''(\xi)$, $\xi\in\cG$. But then the
identity $\gl R^0_\gl = A^0 R^0_\gl + \text{id}$ implies the following formula
\begin{equation*}
    \frac{1}{2}\,\bigl(R^0_\gl f\bigr)''(v) + \gb\,R^0_\gl f(v) = 0.
\end{equation*}
For every $\gl>0$ $R^0_\gl$ maps $\Co$ onto $\cD(A^0)$. With
the choice $a = (1+\gb)^{-1}\,\gb$, $c = (1+\gb)^{-1}$ this shows that the process
realizes the boundary conditions of equations~\eqref{eq1.1} with $b_k=0$, $\kn$.

Moreover, we can now use equation~\eqref{eq1.23} combined with formula~\eqref{eq1.25}, to
obtain the following explicit expression for the resolvent with $f\in\Co$, $\gl>0$:
\begin{equation}    \label{eq1.26}
    R^0_\gl f(\xi) = R^D_\gl f(\xi) + \frac{1}{\gb+\gl}\,e^{-\sgl d(\xi,v)}\,f(v),
                    \qquad \xi\in\cG,
\end{equation}
where, as before, $R^D_\gl$ is the Dirichlet resolvent.

In order to compute the heat kernel associated with this process on $\cG$, we invert
the Laplace transforms in equation~\eqref{eq1.26}. For the first term on the right
hand side this is trivial, and gives the Dirichlet heat kernel $p^D$, cf.\
equation~\eqref{eq1.20}. The second term could be handled by a formula which can be
found in the tables (e.g., \cite[eq.~(5.6.10)]{ErMa54a}). But this formula involves
the complementary error function $\erfc$ at complex arguments, and does not yield a
very intuitive expression. Instead, we can simply use the observation that
$t\mapsto \exp(-\gb t)$ is the inverse Laplace transform of $\gl\mapsto (\gb+\gl)^{-1}$.
Moreover, the well-known formula for the density of the hitting time of the origin
by a Brownian motion on the real line (e.g., \cite[p.~25]{ItMc74}, \cite[p.~96]{KaSh91},
\cite[p.~102]{ReYo91})
gives us the following expression for the density of the first hitting time of the vertex
\begin{equation}	\label{DFT}
	P_\xi(H_v\in ds) = \frac{d(\xi,v)}{\sqrt{2\pi s^3}}\,e^{-d(\xi,v)^2/2s}\,ds,\qquad s\ge 0.
\end{equation}
Using the well-known Laplace transform (e.g., \cite[eq.~4.5.28]{ErMa54a})
\begin{equation}    \label{LT}
    \int_0^\infty e^{-\gl s}\,\frac{a}{\sqrt{2\pi s^3}}\,e^{-a^2/2s}\,ds
        = e^{-\sgl a},\qquad a>0,\,\gl>0,
\end{equation}
we infer that the inverse Laplace transform of the exponential in~\eqref{eq1.26}
is given by $P_\xi(H_v\in ds)$. Thus we obtain the following heat kernel
\begin{equation}    \label{eq1.27}
\begin{split}
    p^0(t,\xi,d\eta)
        &= p^D(t,\xi,\eta)\,d\eta - \Bigl(\int_0^t e^{-\gb(t-s)}
                        \,P_\xi(H_v\in ds)\Bigr)\,\gep_v(d\eta),\\[1ex]
        &= p^D(t,\xi,\eta)\,d\eta - \Bigl(\int_0^t e^{-\gb(t-s)}
                \,\frac{d(\xi,v)}{\sqrt{2\pi s^3}}\,e^{-d(\xi,v)^2/2s}\,ds\Bigr)\,\gep_v(d\eta),
\end{split}
\end{equation}
with $\xi$, $\eta\in\cG$, $t>0$, and $\gep_v$ is the Dirac measure at the vertex $v$.

\subsection{Killing via the Local Time at the Vertex}    \label{ssect1.5}
We recall from remark~\Iref{rem_3_2}, that we may and will consider every Brownian
motion $X$ on $\cG$ with respect to a filtration $\cF = (\cF_t,\,t\ge 0)$ which is
right continuous and complete relative to $(P_\xi,\,\xi\in\cG)$, and such that
$X$ is strongly Markovian with respect to $\cF$.

In this subsection we suppose that $X$ is a Brownian motion on the single vertex
graph $\cG$ with infinite lifetime, and such that the vertex is not absorbing. This
entails (cf.\ section~\Iref{sect_3}) that $X$ leaves the vertex immediately and
begins a standard Brownian excursion into one of the edges. Therefore we get in this
case for the hitting time $H_v$ of the vertex $P_v(H_v=0)=1$, i.e., $v$ is regular
for $\{v\}$ in the sense of~\cite{BlGe68}. Consequently $X$ has a local time
$L=(L_t,\,t\ge 0)$ at the vertex (e.g., \cite[Theorem~V.3.13]{BlGe68}). Without loss
of generality, we suppose throughout this subsection that $L$ is a \emph{perfect
continuous homogeneous additive functional (PCHAF)} of $X$ in the sense
of~\cite[Section~III.32]{Wi79}. That is, $L$ is a non-decreasing process, which is
adapted to $\cF$, and such that it is a.s.\ continuous, additive, i.e., $L_{t+s} =
L_t + L_s\comp \theta_t$, and for all $\xi\in\cG$, $P_\xi\bigl(L_0=0\bigr)=1$ holds
true. Moreover we may and will assume from now on that $X$ and $L$ are
\emph{pathwise} continuous.

Killing $X$ exponentially on the scale of $L$, we can construct a new Brownian
motion $\hX$ on $\cG$.  We shall do this using the method of \cite{Kn81, KaSh91}.

Let $K=(K_s,\,s\in\R_+)$ denote the right continuous pseudo-inverse of $L$:
\begin{equation}    \label{eq_def_K}
    K_s = \inf\{t\ge 0,\,L_t>s\},\qquad s\in\R_+,
\end{equation}
where --- as usual --- we make the convention that $\inf\emptyset=+\infty$. The
continuity of $L$ entails that for every $s\in\R_+$, $L_{K_s}=s$. Clearly, $K$ is
increasing, and due to its right continuity it is a measurable stochastic process.
Fix $s\in\R_+$. It is straightforward to check that for every $t\in\R_+$,
\begin{equation}    \label{eq_set_K}
    \{K_s < t\} = \{L_t > s\}.
\end{equation}
Because $L$ is adapted, the set on the right hand side belongs to $\cF_t$, and since
$\cF$ is right continuous, equality~\eqref{eq_set_K} shows that for every $s\in\R_+$,
$K_s$ is a stopping time relative to $\cF$. We remark that since $L$ only increases
when $X$ is at the vertex $v$, the continuity of $X$ implies that for every
$s\in\R_+$ we get $X(K_s) = v$ on $\{K_s<+\infty\}$. On the other hand, we shall
argue below that $L$ a.s.\ increases to $+\infty$, so that we get $X(K_s)=v$
a.s.\ for all $s\in\R_+$.

Let $\gb>0$. Bring in the additional probability space $(\R_+, \cB(\R_+), P_\gb)$,
where $P_\gb$ is the exponential law with rate $\gb$. Let $S$ denote the
associated coordinate random variable $S(s) = s$, $s\in\R_+$. Define
\begin{align*}
    \hgO &= \gO\times\R_+,\\
    \hcA &= \cA\otimes\cB(\R_+),\\
    \hP_\xi &= P_\xi\otimes P_\gb,\quad \xi\in\cG.
\end{align*}
We extend $X$, $L$, $K$, and $S$ in the canonical way to these enlarged probability
spaces, but for simplicity keep the same notation for these quantities.

Set
\begin{equation}    \label{eq1.28}
    \zeta_\gb = \inf\bigl\{t\ge 0,\,L_t>S\bigr\},
\end{equation}
and observe that since $K$ is measurable we may write $\zeta_\gb = K_S$. Thus as
above we get $X(\zeta_\gb) = v$. Define the \emph{killed process}
\begin{equation}    \label{eq1.29}
    \hX_t =  \begin{cases}
                X_t,    & t<\zeta_\gb,\\
                \gD,    & t\ge \zeta_\gb.
             \end{cases}
\end{equation}

Since this prescription for killing the process $X$ via the (PCHAF) $L$ is slightly
different from the method used in~\cite{BlGe68, Wi79}, we cannot use the results
proved there to conclude that the subprocess $\hX$ of $X$ is still a strong Markov
process. However, it has been proved in~\cite[Appendix~\ref{0_app_killing}]{BMMG0}
that the strong Markov property is preserved under this method of killing, i.e.,
$\hX$ is a strong Markov process relative to its natural filtration (actually
relative to a larger filtration, but we will not use this here). Now we may employ
the arguments in section~2.7 of~\cite{KaSh91}, or in section~III.2 of~\cite{ReYo91}
to conclude that $\hX$ is a strong Markov process with respect to the universal
right continuous and complete augmentation of its natural filtration.

It is clear that $\hX$ has a.s.\ right continuous paths which admit left limits, and
that its paths on $[0,\zeta_\gb)$ are equal to those of $X$, and thus are continuous
on this random time interval. Moreover, it is obvious that for every $\xi\in\cGo$,
we have $P_\xi(\zeta_\gb \ge H_v)=1$. Therefore, up to its hitting time of the vertex,
$\hX$ is equivalent to a Brownian motion on the edge to which $\xi$ belongs, because
so is $X$. Altogether we have proved that --- under the hypothesis that $L$ is a
PCHAF, which will be argued below in all cases that we consider --- $\hX$ is a
Brownian motion on $\cG$ in the sense of definition~\Iref{def_3_1}.

There is a simple, useful relationship between the resolvents $R$ and $\hR$ of the
processes $X$ and $\hX$, respectively. Recall our convention that all functions $f$
on $\cG$ are extended to $\cG\cup\{\gD\}$ by $f(\gD)=0$.

\begin{lemma}   \label{lem1.6}
For all $\gl>0$, $f\in B(\cG)$, $\xi\in\cG$,
\begin{equation}    \label{eq1.30}
    \hR_\gl f(\xi)
        = R_\gl f(\xi) - e_\gl(\xi)\,\hE_v\bigl(e^{-\gl \zeta_\gb}\bigr)\,R_\gl f(v)
\end{equation}
holds true, where $e_\gl$ is defined in equation~\eqref{eq1.14}.
\end{lemma}

\begin{proof}
For $\gl>0$, $f\in B(\cG)$, $\xi\in\cG$
\begin{align*}
    \hR_\gl f(\xi)
        &= \hE_\xi\Bigl(\int_0^{\zeta_\gb} e^{-\gl t} f(X_t)\,dt\Bigr)\\
        &= R_\gl f(\xi) - \hE_\xi\Bigl(e^{-\gl\zeta_\gb} \int_0^\infty
                e^{-\gl t}\,f\bigl(X_{t+\zeta_\gb}\bigr)\,dt\Bigr).
\end{align*}
By construction, the last expectation value is equal to
\begin{equation*}
    \gb\int_0^\infty e^{-\gb s}\int_0^\infty e^{-\gl t}
        E_\xi\Bigl(e^{-\gl K_s}\,f\bigl(X_{t+K_s}\bigr)\Bigr)\,dt\,ds,
\end{equation*}
where we used Fubini's theorem. Consider the expectation value under the integrals,
and recall that for fixed $s\in\R_+$, $K_s$ is an $\cF$--stopping time, while
$X$ is strongly Markovian relative to $\cF$. Hence we can compute as follows
\begin{align*}
    E_\xi\Bigl(e^{-\gl K_s}\,f\bigl(X_{t+K_s}\bigr)\Bigr)
        &= E_\xi\Bigl(e^{-\gl K_s}\,E_\xi\Bigl(f\bigl(X_{t+K_s}\cond \cF_{K_s}\bigr)\Bigr)\Bigr)\\
        &= E_\xi\Bigl(e^{-\gl K_s}\,E_{X(K_s)}\bigl(f(X_t)\bigr)\Bigr)\\
        &= E_\xi\bigl(e^{-\gl K_s}\bigr)\,E_v\bigl(f(X_t)\bigr),
\end{align*}
where we used the fact that a.s.\ $X(K_s)=v$. So far we have established
\begin{equation*}
    \hR_\gl f(\xi) = R_\gl f(\xi) - \hE_\xi\bigl(e^{-\gl \zeta_\gb}\bigr)\,R_\gl f(v).
\end{equation*}
In order to compute the expectation value on the right hand side, we first remark
that because $L$ is zero until $X$ hits the vertex for the first time, we find
that for given $s\in\R_+$, $K_s\ge H_v$, and therefore $K_s = H_v + K_s\comp \theta_{H_v}$.
Hence, and again by the strong Markov property,
\begin{align*}
    E_\xi\bigl(e^{-\gl K_s}\bigr)
        &= E_\xi\bigl(e^{-\gl H_v}\,e^{-\gl K_s\comp H_v}\bigr)\\
        &= E_\xi\bigl(e^{-\gl H_v}\,E_\xi\bigl(e^{-\gl K_s\comp H_v}\cond \cF_{H_v}\bigr)\bigr)\\
        &= E_\xi\bigl(e^{-\gl H_v}\bigr) E_v\bigl(e^{-\gl K_s}\bigr),
\end{align*}
Integrating the last identity against the exponential law in the variable $s$, we
find with formula~\eqref{eq1.14}
\begin{equation*}
    \hE_\xi\bigl(e^{-\gl\zeta_\gb}\bigr) = e_\gl(\xi)\,\hE_v\bigl(e^{-\gl \zeta_\gb}\bigr),
\end{equation*}
and the proof is finished.
\end{proof}

\begin{remark}  \label{rem1.7}
Formula~\eqref{eq1.30} is quite useful, because if the resolvent of $X$ is known, then
--- in view of equation~\eqref{eq1.14} --- it reduces the calculation of $\hR_\gl$
to the computation of the Laplace transform of the density of $\zeta_\gb$ under
$\hP_v$.
\end{remark}

\section{The Walsh Process} \label{sect2}
The most basic process --- which on a single vertex graph plays the same role as a
reflected Brownian motion on the half line --- is the well-known \emph{Walsh
process}, which we denote by $W=(W_t,\,t\ge 0)$. It corresponds to the case where
the parameters $a$ and $c$ in the boundary condition~\eqref{eq1.1} both vanish. This
process has been introduced by Walsh in~\cite{Wa78} as a generalization of the skew
Brownian motion discussed in\cite[Chapter~4.2]{ItMc74} to a process in $\R^2$ which
only moves on rays connected to the origin.

A pathwise construction of the Walsh process in the present context is as follows.
Consider the paths of the standard Brownian motion $B=(B_t,\,t\ge 0)$ on $\R$, and
its associated reflected Brownian motion $|B|=(|B_t|,\,t\ge 0)$, where $|\,\cdot\,|$
denotes absolute value. Let $Z=\{t\ge 0,\,B_t=0\}$. Then its complement $Z^c$ is
open, and hence it is the pairwise disjoint union of a countable family of
\emph{excursion intervals} $I_j = (t_j, t_{j+1})$, $j\in\N$. Let $R=(R_j,\,j\in\N)$
be an independent sequence of identically distributed random variables, independent of $B$,
with values in $\sn$ such that $R_j$, $j\in\N$, takes the value $k\in\sn$ with probability
$w_k\in [0,1]$, $\sum_k w_k =1$. Now define $W_t = v$ if $t\in Z$, and if $t\in I_j$, and
$R_j = k$ set $W_t = (k, |B_t|)$. In other words, when starting at $\xi\in\cGo$, the
process moves as a Brownian motion on the edge containing $\xi$ until it hits the
vertex at time $H_v$, and then $W$ performs Brownian excursions
from the vertex $v$ into the edges $l_k$, $k\in\sn$, whereby the edge $l_k$ is
selected with probability $w_k$.

As for the standard Brownian motion on $\R$ (cf.\ subsection~\ref{ssect1.2}), we
may and will assume without loss of generality that $W$ has exclusively continuous
paths.

Walsh has remarked in the epilogue of~\cite{Wa78}, cf.\ also~\cite{BaPi89}, that it
is not completely straightforward to prove that this stochastic process is strongly
Markovian. A proof of the strong Markov property based on It\^o's excursion
theory~\cite{It71a} has been given in~\cite{Sa86a, Sa86b}. A construction of this
process via its Feller semigroup can be found in~\cite{BaPi89} (cf.\ also the references
quoted there for other approaches).

Next we check that the Walsh process has a generator with boundary condition at the
vertex given by \eqref{eq1.1} with $a=c=0$. Let $f\in\cD(A^w)$. At the vertex $v$
Dynkin's form for the generator reads
\begin{equation}    \label{eq2.1}
    A^w f(v) = \lim_{\gep\da 0} \frac{E_v\Bigl(f\bigl((X(\TW_{v,\gep})\bigr)\Bigr)-f(v)}
                                  {E_v(\TW_{v,\gep})},
\end{equation}
where $\TW_{v,\gep}$ is the hitting time of the complement of the open ball
$B_\gep(v)$ of radius $\gep>0$ around $v$.

\begin{lemma}   \label{lem2.1}
For the Walsh process $E_v(\TW_{v,\gep}) = \gep^2$.
\end{lemma}

\begin{proof}
Since by construction $W$ has infinite lifetime, $\TW_{v,\gep}$ is the hitting time
of the set of the $n$ points with local coordinates $(k,\gep)$, $\kn$. Therefore, by the
independence of the choice of the edge for the values of the excursion, it follows
that under $P_v$ the stopping time $\TW_{v,\gep}$ has the same law as the hitting
time of the point $\gep>0$ of a reflected Brownian motion on $\R_+$, starting at
$0$. Thus the statement of the lemma follows from equation~\eqref{eq1.6}.
\end{proof}

From the construction of $W$ we immediately get
\begin{equation*}
    E_v\Bigl(f\bigl(W(\TW_{v,\gep})\bigr)\Bigr) = \sum_{k=1}^n w_k f_k(\gep),
\end{equation*}
with the notation $f_k(x)=f(k,x)$, $x\in\R_+$. Inserting this into
equation~\eqref{eq2.1} we obtain
\begin{equation*}
    A^w f(v) = \lim_{\gep\da 0} \gep^{-2} \sum_{k=1}^n w_k \bigl(f_k(\gep)-f(v)\bigr),
\end{equation*}
and since $f'(v_k)$ exists (cf.\ section~\Iref{sect_5}) it is obvious that
this entails the condition
\begin{equation}    \label{eq2.2}
    \sum_{k=1}^n w_k f'(v_k) = 0.
\end{equation}
For later use we record this result as

\begin{theorem} \label{thm2.2}
Consider the boundary condition~\eqref{eq1.1} with $a=c=0$, and $b\in [0,1]^n$.
Let $W$ be a Walsh process as constructed above with the choice $w_k=b_k$,
$k\in\sn$. Then the generator $A^w$ of $W$ is $1/2$ times the second derivative on
$\cG$ with domain consisting of those $f\in\Cii$ which satisfy condition~\eqref{eq1.1b}.
\end{theorem}

For the remainder of this section we make the choice $a=c=0$,
$w_k=b_k$, $k\in\sn$ in~\eqref{eq1.1}.

Next we compute the resolvent of $W$. Let $\gl>0$, $f\in\Co$, and consider first
$\xi=v$. Without loss of generality, we may assume that $W$ has been constructed pathwise
from a standard Brownian motion $B$ as described above, and that $B$ is as in
subsection~\ref{ssect1.2}. Then we get
\begin{equation*}
    E_v\bigl(f(W_t)\bigr)
        = \sum_{m=1}^n b_m E^Q_0\bigl(f_m(|B_t|)\bigr).
\end{equation*}
Hence we find for the resolvent $\RW$ of the Walsh process
\begin{subequations}    \label{eq2.3}
\begin{equation}    \label{eq2.3a}
    \RW_\gl f(v)
        = \int_\cG \rW_\gl(v,\eta)\, f(\eta)\,d\eta
\end{equation}
with resolvent kernel $\rW_\gl (v,\eta)$, $\eta\in\cG$, given by
\begin{equation}    \label{eq2.3b}
    \rW_\gl(v,\eta) = \sum_{m=1}^n 2 b_m\, \frac{e^{-\sgl\, d(v,\eta)}}{\sgl}\,\onel[m](\eta),
\end{equation}
\end{subequations}
and where the integration in~\eqref{eq2.3a} is with respect to the Lebesgue
measure on $\cG$.

Now let $\xi\in\cG$. We use the first passage time formula~\eqref{eq1.23} together with
formulae~\eqref{eq1.14} and \eqref{eq2.3}, and obtain

\begin{lemma}   \label{lem2.3}
The resolvent of the Walsh process on $\cG$ is given by
\begin{subequations}    \label{eq2.4}
\begin{equation}    \label{eq2.4a}
    \RW_\gl f(\xi) = \int_\cG \rW_\gl(\xi,\eta) f(\eta)\,d\eta,\qquad \gl>0,\,\xi\in\cG,\,f\in B(\cG),
\end{equation}
with
\begin{align}
    \rW_\gl(\xi,\eta)
        &= r_\gl(\xi,\eta) + \sum_{k,m=1}^n e_{\gl,k}(\xi) \,\SW_{km}
                \,\frac{1}{\sgl}\,e_{\gl,m}(\eta), \label{eq2.4b}\\
    \SW_{km}
        &= 2 w_m - \gd_{km}, \label{eq2.4c}
\end{align}
where $r_\gl$ is defined in equation~\eqref{eq1.22a}, and where $e_{\gl,k}$,
$e_{\gl,m}$ denote the restrictions of $e_\gl$ (cf.~\eqref{eq1.14}) to
the edges $l_k$, $l_m$ respectively.
\end{subequations}
\end{lemma}

\begin{remark}  \label{rem2.4}
The matrix $\SW=\bigl(\SW_{km},\,k,m = 1,\dotsc,n\bigr)$ is the \emph{scattering
matrix} as defined in quantum mechanics. We briefly recall its construction in the
present context, for more details the interested reader is referred
to~\cite{KoSc99}. $\SW$ is obtained from the boundary conditions at the vertex $v$
in the following way. Consider a function $f$ on $\cG$ which is continuously
differentiable in $\cGo=\cG\setdif\{v\}$, and such that for all $\kn$ the limits
\begin{align*}
    F_k  &= f(v_k) = \lim_{\xi\to v,\,\xi\in l_k\op} f(\xi)\\
    F'_k &= f'(v_k) = \lim_{\xi\to v,\,\xi\in l_k\op} f'(\xi)
\end{align*}
exist. Define two column vectors $F$, $F'\in\C^n$, having the components $F_k$ and
$F'_k$, $\kn$, respectively. Furthermore, consider boundary conditions of the following
form
\begin{equation}    \label{bc}
    A F + B F' =0,
\end{equation}
where $A$ and $B$ are complex $n\times n$ matrices. The \emph{on-shell scattering
matrix at energy $E>0$} is defined as
\begin{equation}    \label{SE}
    S_{A,B}(E) = - (A+i\sqrt{E}B)^{-1}(A-i\sqrt{E}B),
\end{equation}
which exists and is unitary, provided the $n\times 2n$ matrix $(A,B)$ has
maximal rank (i.e., rank $n$) and $AB^\dagger$ is hermitian. These requirements for $A$
and $B$ guarantee that the corresponding Laplacean is a self-adjoint operator
on $L^2(\cG)$ (with Lebesgue measure). Observe that under these conditions the
boundary conditions~\eqref{bc} are equivalent to any boundary conditions of the form
$CAF+CBF'=0$ where $C$ is invertible. Also $S_{CA, CB}(E) = S_{A,B}(E)$ holds true.
For the Walsh process at hand, concrete choices for $A$ and $B$ are given by
\begin{equation*}
    A^w=\begin{pmatrix}
            0&0&0&\ldots&0&0\\
            1&-1&0&\ldots&0&0\\
            0&1&-1&\ldots&0&0\\
            0&0&1&\ldots &0&0\\
            \vdots&\vdots&\vdots&\ddots&\vdots&\vdots\\
            0&0&0&\ldots&1&-1
        \end{pmatrix},\ %\\
    B^w=\begin{pmatrix}
            b_1&b_2&b_3&\ldots&b_{n-1}&b_n\\
            0&0&0&\ldots&0&0\\
            0&0&0&\ldots&0&0\\
            \vdots&\vdots&\vdots&\ddots&\vdots&\vdots\\
            0&0&0&\ldots&0&0\\
            0&0&0&\ldots&0&0
        \end{pmatrix}.
\end{equation*}
Then~\eqref{bc} is the condition that $f$ is actually continuous at the vertex $v$,
i.e., $f(v_k)=f(v_m)$, $k$, $m=1,\dotsc,n$, and that~\eqref{eq2.2} is valid (with
$w_k=b_k$, $k\in\sn$). Obviously, $(A^w, B^w)$ has maximal rank. However, $A^w
(B^w)^\dagger$ is hermitian if and only if all $b_k$ are equal (i.e., $b_k = 1/n$,
$\kn$). Nevertheless, \eqref{SE} exists also in the non-hermitian case, and $S_{A^w,
B^w}(E)=S^w$ holds for all $E>0$ due to the relations $A^w S^w = - A^w$, and $B^w
S^w = B^w$. In addition, the following relations are valid:
\begin{align}
         \SW &= \bigl(\SW\bigr)^{-1}, \label{eq2.5}\\
    \det \SW &= (-1)^{n+1}.           \label{eq2.6}
\end{align}
Furthermore, $S^w$ is a contraction, and the associated Laplace operator is
m-dissipative on $L^2(\cG)$ since trivially $\text{Im}(AB^\dagger)=0$, cf.\ Theorem~2.5
in~\cite{KoPo09c}. When all $b_k$ are equal, such that $A^w(B^w)^\dag=0$, then $S^w$ is
an involutive, orthogonal matrix of the form
\begin{equation}	\label{eq2.6a}
	S^w = - \1 + 2P_n.
\end{equation}
$P_n$ is the matrix whose entries are equal to $1/n$.
$P_n$ is a real orthogonal projection, that is $P_n = P_n^\dag = P_n^t = P_n^2$.
It is also of rank $1$, that is $\dim \text{Ran} P_n = 1$.
The relation~\eqref{eq2.4b} giving the resolvent in terms of the
scattering matrix is actually valid in the more general context of arbitrary metric
graphs and boundary conditions of the form~\eqref{bc}, see~\cite{KoSc06, KoPo09c}.
\end{remark}

It is straightforward to compute the inverse Laplace transform of the right hand side
of formula~\eqref{eq2.4b}, and this yields the following result.

\begin{lemma}   \label{lem2.5}
For $t>0$, $\xi$, $\eta\in\cG$ the transition density of the Walsh process on $\cG$ is
given by
\begin{align}
    \pW(t,\xi,\eta)
        &= p^D(t,\xi,\eta) + \sum_{k,m=1}^n \onel(\xi)\,
                2w_m\, g\bigl(t,d_v(\xi,\eta)\bigr)\,\onel[m](\eta), \label{eq2.7}\\
        &= p(t,\xi,\eta) + \sum_{k,m=1}^n \onel(\xi)\,
                \SW_{km} g\bigl(t,d_v(\xi,\eta)\bigr)\,\onel[m](\eta). \label{eq2.8}
\end{align}
$p(t,\xi,\eta)$ is defined in equation~\eqref{eq1.16}, $p^D(t,\xi,\eta)$ in
equation~\eqref{eq1.20}, $g$ is the Gau{\ss}-kernel~\eqref{eq1.4}, and $d_v$ is
defined in equation~\eqref{eq1.15}.
\end{lemma}

\begin{remark}  \label{rem2.6}
Alternatively $\pW(t,\xi,\eta)$ can also be written as
\begin{equation}    \label{eq2.9}
\begin{split}
    \pW(t,\xi,\eta)
        &= p(t,\xi,\eta)\\
        &\  + \sum_{k,m=1}^n \onel(\xi)
            \int_0^t P_\xi(H_v\in ds)\, \SW_{km}\, g\bigl(t-s,d(v,\eta)\bigr)\,\onel[m](\eta).
\end{split}
\end{equation}
Even though this formula appears somewhat more complicated than~\eqref{eq2.8}, it
exhibits the role of the scattering matrix $\SW$, that is, it describes more clearly what
happens when the process hits the vertex.
\end{remark}

\section{The Elastic Walsh Process} \label{sect3}
In this section we consider the boundary conditions~\eqref{eq1.1} with $0<a<1$ and
$c=0$. The corresponding stochastic process, which we will denote by $\We$, is constructed
from the Walsh process $W$ of the previous section in the same way as the elastic Brownian
motion on $\R_+$ is constructed from a reflected Brownian motion (cf., e.g., \cite{ItMc63},
\cite[Chapter~2.3]{ItMc74}, \cite[Chapter~6.2]{Kn81}, \cite[Chapter~6.4]{KaSh91}).

In more detail, the construction is as follows.
Consider the Walsh process $W$ as discussed in the previous section. We may continue
to suppose that $W$ has been constructed pathwise from a standard Brownian motion
$B$, as it has been described there. But then the local time of $W$ at the vertex,
denoted by $L^w$, is pathwise equal to the local time of the Brownian motion at the
origin (and we continue to use the normalization determined by~\eqref{eq1.7a}). It
is well-known (e.g., \cite{ItMc74, Kn81, KaSh91, ReYo91}) that $L^w$ has all
properties of a PCHAF as formulated in subsection~\ref{ssect1.5} for the construction
of a subprocess by killing $W$ at the vertex. We continue to denote the rate of the
exponential random variable $S$ used there by $\gb>0$. Let $\We$ be the subprocess
so obtained. In particular (cf.~\ref{ssect1.5}), $\We$ is a Brownian motion on
$\cG$, and in analogy with the case of a Brownian motion on the real line we call
this stochastic process the \emph{elastic Walsh process}. We write $\zeta_{\gb,0}$
for the lifetime of $\We$ (i.e., for the random time corresponding to $\zeta_\gb$ in
subsection~\ref{ssect1.5}).

We proceed to show that the elastic Walsh process $\We$ has a generator $A^e$ with
domain $\cD(A^e)$
which satisfies the boundary conditions as claimed. In other words, we claim that there
exist $a\in(0,1)$ and $b_k\in (0,1)$, $k\in\sn$, with $a + \sum_k b_k =1$, so that
for all $f\in\cD(G)$,
\begin{equation}    \label{eq3.1}
    a f(v) = \sum_{k=1}^n b_k f'(v_k)
\end{equation}
holds. To this end, we calculate $A^e f(v)$ in Dynkin's form. We shall use a notation
similar to the one used in subsection~\ref{ssect1.5}. Namely, let $\hP_v$ and $\hE_v$
denote the probability and expectation, respectively, on the probability space extended by
$(\R_+,\cB(\R_+),P_\gb)$, while the corresponding symbols without $\hat{\,\cdot\,}$
are those for the Walsh process without killing.

For $\gep>0$ and under $\We$ let $\TeW_{v,\gep}$ denote the hitting time of the
complement $B_\gep(v)^c$  of the open ball $B_\gep(v)$ of radius $\gep>0$ with center $v$. Then
$\TeW_{v,\gep} = \TW_{v,\gep} \land \zeta_{\gb,0}$, where as before $\TW_{v,\gep}$
is the hitting time of $B_\gep(v)^c$ by the Walsh process $W$. (Note that $B_\gep(v)^c$
contains the cemetery $\gD$.) We find
\begin{equation}    \label{eq3.1a}
    \hE_v\Bigl(f\bigl(\We(\TeW_{v,\gep})\bigr)\Bigr)
        = \Bigl(\sum_{k=1}^n w_k f_k(\gep)\Bigr)\,\hP_v\bigl(\TW_{v,\gep}<\zeta_{\gb,0}\bigr).
\end{equation}
The probability in the last expression is taken care of by the following lemma.

\begin{lemma}   \label{lem3.1}
For all $\gep$, $\gb>0$,
\begin{equation}    \label{eq3.2}
    \hP_v\bigl(\TW_{v,\gep} < \zeta_{\gb,0}\bigr) = \frac{1}{1+\gep\gb}.
\end{equation}
\end{lemma}

\begin{proof}
We may consider the Walsh process $W$ as being pathwise constructed from a standard
Brownian motion $B$ on the real line as in the previous section, and we shall use
the notations and conventions from there. Then it is clear that under $P_v$ and
under $\hP_v$, $\TW_{v,\gep}$ has the same law as the hitting time of the point
$\gep$ in $\R_+$ by the reflecting Brownian motion $|B|$ under $Q_0$, that is, as
$H^B_{\{-\gep,\gep\}}$ of the Brownian motion $B$ itself under $Q_0$. Let $K^w$ denote
the right continuous pseudo-inverse of $L^w$. For fixed $s\in\R_+$ we get
\begin{equation*}
    \{K^w_s < \TW_{v,\gep}\} = \{L^w(\TW_{v,\gep})>s\}.
\end{equation*}
Hence
\begin{align*}
    P_v\bigl(K^w_s < \TW_{v,\gep}\bigr)
        &= P_v\bigl(L^w(\TW_{v,\gep})>s\bigr)\\
        &= Q_0\bigl(L^B(H^B_{\{-\gep,\gep\}})>s\bigr).
\end{align*}
In appendix~\ref{0_app_LT} of~\cite{BMMG0} is shown with the method
in~\cite[Section~6.4]{KaSh91} that under $Q_0$ the random variable
$L^B(H^B_{\{-\gep,\gep\}})$ is exponentially distributed with mean $\gep$. So we find
\begin{equation*}
    P_v\bigl(K^w_s < \TW_{v,\gep}\bigr) = e^{-s/\gep}.
\end{equation*}
We integrate this relation against the exponential law with rate $\gb$ in the variable $s$,
and obtain
\begin{align*}
    \hP_v\bigl(\zeta_{\gb,0}>\TW_{v,\gep}\bigr)
        &= 1 - \gb\int_0^\infty e^{-\gb s}\,P_v\bigl(K^W_s < \TW_{v,\gep}\bigr)\,ds\\
        &= \frac{1}{1+\gep\gb}.
\end{align*}
We used the fact that due to the continuity of the paths of $W$ we have
$\zeta_{\gb,0}\ne \TW_{v,\gep}$.
\end{proof}

We insert formula~\eqref{eq3.2} into equation~\eqref{eq3.1a}, and obtain
\begin{align*}
    A^e f(v)
        &=  \lim_{\gep\da 0}\, \frac{1}{\hE_v(\TeW_{v,\gep})}
                \,\Bigr(\hE_v\Bigl(f\bigl(\We(\TeW_{v,\gep})\bigr)\Bigr) - f(v)\Bigl)\\
        &=  \lim_{\gep\da 0}\, \frac{1}{\hE_v(\TeW_{v,\gep})}
                \,\Bigl(\frac{1}{1+\gep\gb}\,\sum_{k=1}^n w_k f_k(\gep) - f(v)\Bigr)\\
        &=  \lim_{\gep\da 0}\, \frac{\gep}{\hE_v(\TeW_{v,\gep})}
                \frac{1}{1+\gep\gb}\,\Bigl(\sum_{k=1}^n w_k\,\frac{f_k(\gep)-f(v)}{\gep}
                    -\gb f(v)\Bigr).
\end{align*}
Obviously $\hE_v(\TeW_{v,\gep})\le E_v(\TW_{v,\gep})=\gep^2$ (cf.~lemma~\ref{lem2.1}).
Since the last limit and $f'(v_k)$, $k\in\sn$, exist and are finite, we get as a necessary
condition
\begin{equation}    \label{eq3.3}
    \sum_{k=1}^n w_k f'(v_k) - \gb f(v) =0.
\end{equation}
Thus we have proved the following theorem.

\begin{theorem} \label{thm3.2}
Consider the boundary condition~\eqref{eq1.1} with $a\in(0,1)$,
$b\in [0,1]^n$, and $c=0$. Set
\begin{equation}    \label{eq3.4}
    w_k = \frac{b_k}{1-a},\,\kn,\quad \gb = \frac{a}{1-a},
\end{equation}
and let $\We$ be the elastic Walsh process as constructed above with these
parameters. Then the generator $A^e$ of $\We$ is $1/2$ times the second derivative
on $\cG$ with
domain consisting of those $f\in\Cii$ which satisfy condition~\eqref{eq1.1b}.
\end{theorem}

\begin{remark}  \label{rem3.3a}
Note that condition~\eqref{eq1.1a} entails that if $w_k$ and $\gb$ are defined
by~\eqref{eq3.4} then $w_k\in [0,1]$, $\kn$, $\sum_k w_k=1$, and $\gb>0$. Therefore
the choice~\eqref{eq3.4} is consistent with the conditions on these parameters
required by the construction of the elastic Walsh process $\We$.
\end{remark}

Next we compute the resolvent $\ReW$ of the elastic Walsh process. As a byproduct this
will give another proof of theorem~\ref{thm3.2}. Moreover, it will provide us
with an explicit formula for the scattering matrix in this case. In contrast to the
calculations in~\cite[Chapter~2.3]{ItMc74}, \cite[Chapter~6.2]{Kn81} for the
classical case with $\cG=\R_+$, we do not use the first passage time
formula~\eqref{eq1.23}, but instead we use formula~\eqref{eq1.30}. This simplifies
the computation considerably.

Let $f\in\Co$, $\gl>0$, and $\xi\in\cG$. In the present context
formula~\eqref{eq1.30} reads
\begin{equation*}
    R^e_\gl f(\xi)
        = \RW_\gl f(\xi) - e_\gl(\xi)\,\hE_v\bigl(e^{-\gl \zeta_{\gb,0}}\bigr)\, \RW_\gl f(v),
\end{equation*}
where $\RW$ is the resolvent of the Walsh process without killing, and $e_\gl$ is
defined in~\eqref{eq1.14}. The Laplace transform of the density of $\zeta_{\gb,0}$ under
$\hP_v$ is readily computed:

\begin{lemma}   \label{lem3.3}
For all $\gl$, $\gb>0$,
\begin{equation*}
    \hE_v\bigl(e^{-\gl \zeta_{\gb,0}}\bigr) = \frac{\gb}{\gb + \sgl}.
\end{equation*}
\end{lemma}

\begin{proof}
As remarked before, we may consider $L^w$ to be equal to the local time at the
origin of the Brownian motion $B$ underlying the construction of $W$, and therefore
the analogous statement is true for the right continuous pseudo-inverse $K^w$ of $L^w$.
As above let $K^B$ denote the right continuous pseudo-inverse of $L^B$ (cf.~\ref{ssect1.2}).
Then for $s\in\R_+$,
\begin{align*}
    E_v\bigl(e^{-\gl K^w_s}\bigr)
        &= E^Q_0\bigl(e^{-\gl K^B_s}\bigr)\\
        &= e^{-\sgl s},
\end{align*}
where we used lemma~\ref{lem1.2}. Hence
\begin{equation*}
    \hE_v\bigl(e^{-\gl \zeta_{\gb,0}}\bigr)
        = \gb \int_0^\infty e^{-(\gb+\sgl) t}\,dt,
\end{equation*}
which proves the lemma.
\end{proof}

With lemma~\ref{lem3.3} we obtain the following formula
\begin{equation}    \label{eq3.5}
    R^e_\gl f(\xi)
        = \RW_\gl f(\xi) - \frac{\gb}{\gb + \sgl}\,e_\gl(\xi)\, \RW_\gl f(v).
\end{equation}
Note that $\RW_\gl f$ is in the domain of the generator of the Walsh process, and
therefore satisfies the boundary condition~\eqref{eq2.2}:
\begin{equation*}
    \sum_{k=1}^n w_k\bigl(\RW_\gl f\bigr)'(v_k) = 0.
\end{equation*}
On the other hand, we obviously have $e_\gl'(v_k) = -\sgl$ for all $k\in\sn$.
Thus with $\sum_{k=1}^n w_k =1$ we find,
\begin{equation*}
    \sum_{k=1}^n w_k \bigl(R^e_\gl f\bigr)'(v_k)
                    = \gb\,\frac{\sgl}{\gb+\sgl}\,\RW_\gl f(v),
\end{equation*}
while equation~\eqref{eq3.5} yields for $\xi=v$
\begin{equation*}
    R^e_\gl f(v) = \frac{\sgl}{\gb+\sgl}\,\RW_\gl f(v).
\end{equation*}
The last two equations show that for all $f\in\Co$, $\gl>0$, we have
\begin{equation*}
    \sum_{k=1}^n w_k \bigl(R^e_\gl f\bigr)'(v_k) = \gb\, R^e_\gl f(v).
\end{equation*}
Since for every $\gl>0$, $R^e_\gl$ maps $\Co$ onto the domain of the generator
of $\We$, we have another proof of theorem~\ref{thm3.2}.

Upon insertion of the expressions for the resolvent kernels of the Walsh process,
equations~\eqref{eq2.3}, and \eqref{eq2.4}, with the same notation as in
lemma~\ref{lem2.3} we immediately obtain the following result:

\begin{lemma}   \label{lem3.4}
For $\gl>0$, $\xi$, $\eta\in\cG$ the resolvent kernel of the elastic Walsh process $\We$
is given by
\begin{subequations}    \label{eq3.6}
\begin{align}
    \reW_\gl(\xi,\eta)
        &= r^D_\gl(\xi,\eta) + \sum_{k,m=1}^n e_{\gl,k}(\xi)\,2w_m
            \,\frac{1}{\gb+\sgl}\,e_{\gl,m}(\eta)  \label{eq3.6a}\\
        &= r_\gl(\xi,\eta) + \sum_{k,m=1}^n e_{\gl,k}(\xi)\,\SeW_{km}(\gl)
                            \,\frac{1}{\sgl}\, e_{\gl,m}(\eta),\label{eq3.6b}
\end{align}
with the scattering matrix $\SeW$
\begin{equation}    \label{eq3.6c}
    \SeW_{km}(\gl) = 2\,\frac{\sgl}{\gb+\sgl}\,w_m - \gd_{km},\qquad \gl>0,\,k,m\in\sn.
\end{equation}
\end{subequations}
\end{lemma}

\begin{remark}  \label{rem3.5}
Note that in contrast to the case of the Walsh process, this time the scattering
matrix is not constant with respect to $\gl>0$. Also, when $\gb=0$,
formula~\eqref{eq2.4c} is recovered, as it should be. In analogy with the discussion
in remark~\ref{rem2.4}, the boundary conditions for the elastic Walsh process is
given by the matrices
\begin{equation*}
    A^e = \begin{pmatrix}
             0&0&0&\ldots&0&\gb\\
             1&-1&0&\ldots&0&0\\
             0&1&-1&\ldots&0&0\\
             0&0&1&\ldots &0&0\\
             \vdots&\vdots&\vdots&\ddots&\vdots&\vdots\\
             0&0&0&\ldots&1&-1
         \end{pmatrix},\ %\\
    B^e=\begin{pmatrix}
            w_1&w_2&w_3&\ldots&w_{n-1}&w_n\\
            0&0&0&\ldots&0&0\\
            0&0&0&\ldots&0&0\\
            0&0&0&\ldots&0&0\\
            \vdots&\vdots&\vdots&\ddots&\vdots&\vdots\\
            0&0&0&\ldots&0&0
        \end{pmatrix},
\end{equation*}
such that
\begin{align*}
    \SeW(\gl) &= S_{A^e,B^e}(E=-2\gl)\\
              &= - (A^e+\sgl B^e)^{-1}\,(A^e-\sgl B^e).
\end{align*}
Observe that for $k$, $m\in\{1,\dotsc,n\}$ the matrix element $\SeW_{km}(\gl)$
of the scattering matrix is obtained from the resolvent kernel as
\begin{equation*}
    \SeW_{km}(\gl)
        = \sgl\,\lim_{\xi,\,\eta\to v}\bigl(\reW_\gl(\xi,\eta) - r_\gl(\xi,\eta)\bigr),
\end{equation*}
where the limit on the right hand side is taken in such a way that $\xi$, $\eta$ converge
to $v$ along the edges $l_k$, $l_m$ respectively. $\SeW_{km}(\gl)$ in turn fixes the
data $w_m$ and $\gb$, e.g., via the behavior for large $\gl$, that is from the behavior
at ``large energies''
\begin{equation*}
    w_m = \frac{1}{2}\bigl(\gd_{km} + \lim_{\gl'\ua\infty} \SeW_{km}(\gl')\bigr),
            \quad \text{for all $k\in\sn$},
\end{equation*}
and
\begin{equation*}
    \gb = \sgl\Bigl(\frac{\gd_{km}+\lim_{\gl'\ua\infty}\SeW_{km}(\gl')}
            {\gd_{km}+\SeW_{mm}(\gl)}-1\Bigr),\quad \text{for all $\gl$, and all $k$, $m\in\sn$}.
\end{equation*}
Alternatively, the data can be obtained from the small $\gl$ behavior, that is the
threshold behavior, of the scattering matrix, since from
\begin{equation*}
    \frac{w_m}{\gb}
        = \lim_{\gl \da 0} \frac{1}{2\sgl}\bigl(\SeW_{km}(\gl)+\gd_{km}\bigr)
            \qquad \text{for all $k\in\sn$},
\end{equation*}
we obtain
\begin{equation*}
    \gb^{-1} = \frac{1}{2\sgl}\Bigl(\sum_{m=1}^n\SeW_{km}(\gl)+1\Bigr)
            \qquad \text{for all $k\in\sn$},
\end{equation*}
and therefore
\begin{equation*}
    w_m = \frac{\lim_{\gl\da 0} \gl^{-1/2}\bigl(\SeW_{km}(\gl)+\gd_{km}\bigr)}
            {\lim_{\gl\da 0} \gl^{-1/2}\Bigl(\sum_m\SeW_{k'm}(\gl)+1\Bigr)}
            \qquad \text{for all $k$, $k'\in\sn$}.
\end{equation*}

Furthermore we remark that in the context of quantum mechanics in the self-adjoint
case $w_k=1/n$, $\kn$, the boundary conditions of the elastic Walsh process are
interpreted as the presence of a $\gd$--potential of strength $\gb$ at the vertex.
\end{remark}

In order to compute expressions for the transition kernel of the elastic Walsh
process, we use the following two inverse Laplace transforms which follow from
formulae (5.3.4) and (5.6.12) in~\cite{ErMa54a} (cf.\ also appendix~\ref{0_app_LTHK}
in~\cite{BMMG0}) ($\gl>0$, $t\ge 0$, $x\ge0$):
\begin{align}
    \frac{\sgl}{\gb+\sgl}\quad
        &\mathop{\longrightarrow}^{\cL^{-1}}\quad
            \gep_0(dt) - \gb\Bigl(\frac{1}{\sqrt{2\pi t}}-\frac{\gb}{2}e^{\gb^2 t/2}
            \erfc\Bigl(\gb\sqrt{\frac{t}{2}}\Bigr)\Bigr)\,dt,\label{eq3.7}\\
    \frac{1}{\gb+\sgl}\,e^{-\sgl x}\quad
        &\mathop{\longrightarrow}^{\cL^{-1}}\quad
            g(t,x) - \frac{\gb}{2}\,e^{\gb x+ \gb^2 t/2}
            \erfc\Bigl(\frac{x}{\sqrt{2t}}+\gb\sqrt{\frac{t}{2}}\Bigr).\label{eq3.8}
\end{align}
Then the inverse Laplace transform of the scattering matrix $\SeW$ is given by
the following measures on $\R_+$:
\begin{equation}    \label{eq3.9}
 \begin{split}
    \seW_{km}(dt)
        = (2w_m-\gd_{km})\,\gep_0(dt) &- 2w_m\gb\,\frac{1}{\sqrt{2\pi t}}\,dt\\
        & + w_m\gb^2 e^{\gb^2t/2}\,\erfc\Bigl(\gb\sqrt{\frac{t}{2}}\Bigr)\,dt,
\end{split}
\end{equation}
with $k$, $m\in\sn$. Moreover, for $t>0$, $x\ge 0$, let us introduce
\begin{equation}    \label{eq3.10}
    g_{\gb,0}(t,x) = g(t,x) - \frac{\gb}{2}\,e^{\gb x+ \gb^2 t/2}
            \erfc\Bigl(\frac{x}{\sqrt{2t}}+\gb\sqrt{\frac{t}{2}}\Bigr).
\end{equation}

\begin{lemma}   \label{lem3.6}
For $t>0$, $\xi$, $\eta\in\cG$, the transition density $\peW$ of the elastic Walsh process
is given by
\begin{equation} \label{eq3.11}
    \peW(t,\xi,\eta)
        = p^D(t,\xi,\eta) + \sum_{k,m=1}^n \onel(\xi)\, 2w_m\,g_{\gb,0}\bigl(t,d_v(\xi,\eta)\bigr)\,
            \onel[m](\eta),
\end{equation}
and alternatively by
\begin{equation}    \label{eq3.12}
\begin{split}
    \peW(t,\xi,\eta)
        = p(t,\xi,\eta) + \sum_{k,m=1}^n &\onel(\xi)\,\Bigl(\int_0^t P_\xi(\TW_v\in ds)\\
            &\times \Bigl(\seW_{km}*g\bigl(\cdot,d(v,\eta)\bigr)\Bigr)(t-s)\,
                \Bigr)\,\onel[m](\eta),
\end{split}
\end{equation}
where $*$ denotes convolution.
\end{lemma}

\section{The Walsh Process with a Sticky Vertex}  \label{sect4}
In this section  we construct Brownian motions on $\cG$ with $a=0$ in the boundary
condition~\eqref{eq1.1}.

Consider the Walsh process $W$ on $\cG$ from section~\ref{sect2} together with
a right continuous, complete filtration $\cFW$, relative to which it is strongly
Markovian. Furthermore, we denote its local time at the vertex $v$ by $L^w$ (cf.\
section~\ref{sect3}).

Again we follow closely the recipe given by It\^o and McKean in~\cite{ItMc63} (cf.\
also \cite[Section~6.2]{Kn81}) for the case of a Brownian motion on the half line.
For $\gg\ge 0$ introduce a new time scale $\tau$ by
\begin{equation}    \label{eq4.1}
    \tau^{-1}: t\mapsto t + \gg L^w_t,    \qquad t\ge 0.
\end{equation}
Since $L^w$ is non-decreasing, $\tau^{-1}$ is strictly increasing. Moreover, we have
$\tau^{-1}(0)=0$ and $\lim_{t\to+\infty}\tau^{-1}(t) = +\infty$, which implies that
$\tau$ exists, and is strictly increasing from $\R+$ onto $\R_+$, too. As is shown
in~\cite[p.~160]{Kn81}, the additivity of $L^w$ entails the additivity of $\tau$ on
its own time scale, i.e.:

\begin{lemma}   \label{lem4.1}
For all $s$, $t\ge 0$, a.s.\ the following formula holds true
\begin{equation}    \label{eq4.2}
    \tau(s+t) = \tau(s) + \tau(t)\comp \theta_{\tau(s)}.
\end{equation}
\end{lemma}

It is easily checked that for every $t\ge 0$, $\tau(t)$ is an $\cFW$--stopping time,
and since $\tau$ is increasing, we obtain the subfiltration $\cFsW = (\cFsW_t,\,t\ge
0)$ of $\cFW$ defined by $\cFsW_t = \cFW_{\tau(t)}$, $t\in\R_+$. Moreover, we set
$\cFW_\infty = \gs(\cFW_t,\,t\in\R_+)$ and $\cFsW_\infty = \gs(\cFsW_t,\,t\in\R_+)$,
and find $\cFsW_\infty\subset \cFW_\infty$. Standard calculations show that the
completeness and the right continuity of $\cFW$ entail the same properties for
$\cF^s$. (For details of the argument in the case where $\cG=\R_+$ we refer the
interested reader to section~\ref{0_sect_st} of~\cite{BMMG0}.)

Define a stochastic process $\Ws$ on $\cG$, called \emph{Walsh process with sticky
vertex}, by
\begin{equation}    \label{eq4.3}
    \Ws_t = W_{\tau(t)},\qquad t\in\R_+.
\end{equation}
Observe that when $W$ is away from the vertex, $L^w$ is constant, and therefore in
this case $\tau^{-1}$ grows with rate $1$. On the other hand, when $W$ is at the
vertex, $\tau^{-1}$ grows faster than with rate $1$, and therefore $\tau$ increases
slower than the deterministic time scale $t\mapsto t$. Thus $\Ws$ ``experiences a
slow down in time'' until $W$ has left the vertex. In this heuristic sense the
vertex is ``sticky'' for $\Ws$, because it spends more time there than $W$.

Note that because $L^w$ has continuous paths, $\tau^{-1}$ and therefore also $\tau$
are pathwise continuous. Consequently, $\Ws$ has continuous sample paths. Since $W$
has continuous paths, it is a measurable process, and hence for every $t\ge 0$,
$W_{\tau(t)}$ is $\cFW_{\tau(t)}$--measurable, that is, $\Ws$ is $\cFsW$--adapted.
Set $\theta^s_t = \theta_{\tau(t)}$. With the additivity~\eqref{eq4.2} of $\tau$ we
immediately find
\begin{equation} \label{sshift}
    W^s_s\comp\theta^s_t = W^s_{s+t},\qquad s,\,t\in\R_+.
\end{equation}
Thus $\theta^s=(\theta^s_t,\,t\in\R_+)$ is a family of shift operators for $\Ws$.

Next we show the strong Markov property of $\Ws$ relative to $\cFsW$ following the
argument sketched briefly in section~6.2 of~\cite{Kn81} for the case $\cG=\R_+$.
First we prove the simple Markov property of $\Ws$ with respect to $\cFsW$. To this
end, let $s$, $t\ge 0$, $\xi\in\cG$, and $C\in\cB(\cG)$. Then we get
with~\eqref{sshift}
\begin{align*}
    P_\xi\bigl(\Ws_{t+s}\in C\cond \cFsW_t\bigr)
        &= P_\xi\bigl(\Ws_s\comp\theta^s_t\in C\cond \cFsW_t\bigr)\\
        &= P_\xi\bigl(W_{\tau(s)}\comp\theta_{\tau(t)}\in C\cond \cFW_{\tau(t)}\bigr)\\
        &= P_{W_{\tau(t)}} \bigl(W_{\tau(s)}\in C\bigr)\\
        &= P_{\Ws_t}\bigl(\Ws_s\in C\bigr),
\end{align*}
where we used the strong Markov property of $W$ with respect to $\cFW$. As a next
step we prove that $\Ws$ has the strong Markov property for its hitting time $H^s_v$
of the vertex. By construction, $\Ws$ and $W$ have the same paths up to the hitting
time of the vertex, and in particular  $H^s_v$ is also the hitting time of the vertex
by $W$, that is, $H^s_v=H_v$. Moreover, since $L^w(H_v)=0$, we get that
$\tau^{-1}(H_v) = H_v = \tau(H_v)$, as well as $\theta^s(H_v) = \theta(H_v)$.
Assume now that $t\ge 0$, $\xi\in\cG$, and $C\in\cB(\cG)$. Then on $\{H_v<+\infty\}$
we can compute with the strong Markov property of $W$ as follows
\begin{align*}
    P_\xi\bigl(\Ws_{t+H_v}\in C \cond \cFW_{H_v}\bigr)
        &= P_\xi\bigl(\Ws_t\comp\theta^s_{H_v}\in C\cond \cFW_{H_v}\bigr)\\
        &= P_\xi\bigl(W_{\tau(t)}\comp\theta_{H_v}\in C\cond  \cFW_{H_v}\bigr)\\
        &= P_v\bigl(W_{\tau(t)}\in C\bigr)\\
        &= P_v\bigl(\Ws_t\in C\bigr).
\end{align*}
It is readily checked that $\cFsW_{H_v}\subset\cFW_{H_v}$, and therefore we get
in particular the strong Markov property of $W^s$ with respect to $H^s_v=H_v$ in the form
\begin{equation}    \label{sSMT_v}
    P_\xi\bigl(\Ws_{t+H^s_v}\in C \cond \cFsW_{H_v}\bigr)
        = P_v\bigl(\Ws_t\in C\bigr).
\end{equation}
Finally, with the strong Markov property of the standard one-dimensional Brownian
motions on every edge and the strong Markov property~\eqref{sSMT_v} just proved
we can apply the arguments of the proof of theorem~\Iref{thm_4_3} to conclude
that $W^s$ is a Feller process. Hence it is strongly Markovian relative to the
filtration $\cF^s$.

By construction, $W^s$ is up to time $H^s_v$ equivalent to a standard
one-dimensional Brownian motion, and it has continuous sample paths. Hence,
altogether we have shown that $\Ws$ is a Brownian motion on $\cG$ in the sense of
definition~\Iref{def_3_1}.

Now we want to compute the generator of $\Ws$, and first we argue that $v$ is not a
trap for $\Ws$. To this end, we may consider $W$ as constructed from a standard
Brownian motion $B$ as described in section~\ref{sect2}. Let $Z$ denote the zero set
of $B$. Given $s\ge 0$ we can choose $t_0\ge s$ in the complement $Z^c$ of $Z$.
Consider $t = \tau^{-1}(t_0)$, i.e., $t = t_0 + \gg L^w_{t_0}$. Obviously $t\ge s$,
and $\tau(t)\in Z^c$. Therefore $B_{\tau(t)} \ne 0$, and consequently
$\Ws_t=W_{\tau(t)}\ne v$.

\begin{theorem} \label{thm4.2}
Consider the boundary condition~\eqref{eq1.1} with $a=0$, $c\in(0,1)$,
and $b\in [0,1]^n$. Set
\begin{equation}    \label{eq4.4}
    w_k = \frac{b_k}{1-c},\,\kn,\qquad \gg = \frac{c}{1-c},
\end{equation}
and let $\Ws$ be the sticky Walsh process as constructed above with these parameters.
Then the generator $A^s$ of $\Ws$ is $1/2$ times the second derivative on $\cG$ with domain
consisting of those $f\in\Cii$ which satisfy condition~\eqref{eq1.1b}.
\end{theorem}

Before we prove theorem~\ref{thm4.2} we first prepare two preliminary results. Let
$\gep>0$, and let $\TsW_{v,\gep}$ denote the hitting time of the complement of the
open ball $B_\gep(v)$ with radius $\gep$ and center $v$ by $\Ws$. Recall that
$\TW_{v,\gep}$ denotes the corresponding first hitting time for the Walsh process
$W$.

\begin{lemma}   \label{lem4.3}
$P_v$--a.s., the formula
\begin{equation}    \label{eq4.5}
    \TsW_{v,\gep} = \TW_{v,\gep} + \gg\,L^w_{\TW_{v,\gep}}
\end{equation}
holds true.
\end{lemma}

\begin{proof}
Let $W$, and therefore also $\Ws$, start in the vertex $v$. Since $\Ws$ and $W$
have continuous paths with infinite lifetime we have for all $\gg\ge 0$
\begin{equation*}
    \TsW_{v,\gep} = \inf\bigl\{t>0,\, d(v,W_{\tau(t)})=\gep\bigr\},
\end{equation*}
and in particular for $\gg=0$,
\begin{equation*}
    \TW_{v,\gep} = \inf\bigl\{t>0,\, d(v,W_t)=\gep\bigr\}.
\end{equation*}
Moreover, as argued above, both infima are a.s.\ finite. Set
\begin{equation*}
    \gs = \TW_{v,\gep} + \gg\,L^w_{\TW_{v,\gep}}.
\end{equation*}
Then $\tau(\gs) = \TW_{v,\gep}$, and therefore
\begin{align*}
    d\bigl(v, \Ws_\gs\bigr)
        &= d\bigl(v, W_{\tau(\gs)}\bigr)\\
        &= d\bigl(v, W_{\TW_{v,\gep}}\bigr)\\
        &= \gep.
\end{align*}
Consequently we get $\TsW_{v,\gep} \le \gs$. To derive the converse inequality
we remark that
\begin{align*}
    \gep
        &= d\bigl(v, \Ws_{\TsW_{v,\gep}}\bigr)\\
        &= d\bigl(v, W_{\tau(\TsW_{v,\gep})}\bigr),
\end{align*}
which implies
\begin{equation*}
    \tau\bigl(\TsW_{v,\gep}\bigr)\ge \TW_{v,\gep}.
\end{equation*}
Since $\tau$ is strictly increasing this entails
\begin{equation*}
    \TsW_{v,\gep} \ge \tau^{-1}\bigl(\TW_{v,\gep}\bigr) = \gs,
\end{equation*}
and the proof is finished.
\end{proof}

\begin{corollary}   \label{cor4.4}
For every $\gg\ge 0$,
\begin{equation}    \label{eq4.6}
    E_v\bigl(\TsW_{v,\gep}\bigr)=\gep^2+\gg\gep
\end{equation}
holds.
\end{corollary}

\begin{proof}
By construction, the paths of $W$ starting in $v$ hit the complement of $B_\gep(v)$
exactly when the underlying standard Brownian motion $B$ (cf.\ section~\ref{sect2})
starting at the origin hits one of the points $\pm\gep$ on the real line. Thus under $P_v$,
$L^w(\TW_{v,\gep})$ has the same law as $L^B(H^B_{\{-\gep,\gep\}})$ under $P_0$.
Lemma~\ref{lem1.3} states that under $P_0$ this random variable
is exponentially distributed with mean $\gep$. Then equation~\eqref{eq4.6} follows
directly from lemmas~\ref{lem4.3}, and \ref{lem2.1}.
\end{proof}

Given these results, we  come to the

\begin{proof}[Proof of theorem~\ref{thm4.2}]
Let $w_k$, $\kn$, and $\gg$ be defined as in~\eqref{eq4.4}, and note that due to the
condition~\eqref{eq1.1a} on $b_k$, $\kn$, and $c$, we have $w_k\in [0,1]$, $\kn$,
$\sum_k w_k =1$, as well as $\gg>0$. Hence we can construct the associated sticky
Walsh process $\Ws$ as above.

Let $A^s$ denote the generator of $\Ws$ with domain $\cD(A^s)$. Then we
have for $f\in\cD(A^s)$, $A^s f(v) = 1/2 f''(v)$ (cf.\ lemma~\Iref{lem_5_2}). On the
other hand, we can compute $A^s f(v)$ via Dynkin's formula as follows
\begin{align*}
    A^s f(v)
        &= \lim_{\gep\da 0} \frac{E_v\Bigl(f\bigl(\Ws(\TsW_{v,\gep})\bigr)\Bigr)-f(v)}
                                 {E_v\bigl(\TsW_{v,\gep}\bigr)}\\[1ex]
        &= \lim_{\gep\da 0} \frac{\sum_k w_k f_k(\gep) - f(v)}{\gep^2 + \gg\gep},
\end{align*}
where we used corollary~\ref{cor4.4}. Since the directional derivatives of $f$ at $v$
\begin{equation*}
    f'(v_k)=\lim_{a\to v,\,a\in l_k}    \frac{f(a)-f(v)}{d(a,v)},\qquad k\in\sn,
\end{equation*}
exist (cf.\ section~\Iref{sect_5}), we obviously get the boundary condition
\begin{equation}    \label{eq4.7}
    \frac{1}{2}\,f''(v) = \frac{1}{\gg}\,\sum_{k=1}^n w_k f'(v_k)
\end{equation}
as a necessary condition. Finally, inserting of the values~\eqref{eq4.4} of the
parameters $w_k$, $\kn$, and $\gg$ into equation~\eqref{eq4.7} we complete the proof
of theorem~\ref{thm4.2}.
\end{proof}

Next we shall compute the resolvent $\RsW$ of the Walsh process with sticky vertex.
Similarly to the alternative proof of theorem~\ref{thm3.2} for the elastic Walsh
process, as a byproduct we obtain an alternative proof of theorem~\ref{thm4.2}. We
begin with the following

\begin{lemma}   \label{lem4.5}
Let $\gl>0$, $f\in \Co$. Then
\begin{equation}    \label{eq4.8}
    \frac{1}{2} \bigl(\RsW_\gl f\bigr)''(v)
        = \frac{1}{\sgl +\gg\gl}\,\Bigl(2\gl\,(e^w_\gl,f) - \sgl f(v)\Bigr)
\end{equation}
holds, where
\begin{equation}    \label{eq4.9}
    e^w_\gl(\xi) = w_k\,e_\gl(\xi),\qquad \xi\in l_k,\,\kn,
\end{equation}
and $e_\gl$ is defined in equation~\eqref{eq1.14}.
\end{lemma}

\begin{proof}
Let $A^s$ be the generator of $\Ws$ on $\Co$. From the identity $A^s \RsW_\gl =
\gl\,\RsW_\gl -\text{id}$, and the definition of $\tau$ we get
\begin{align*}
    \frac{1}{2}\,\bigl(\RsW_\gl f\bigr)''(v)
        &= \gl E_v\Bigl(\int_0^\infty e^{-\gl t}\bigl(f(\Ws_t)-f(v)\bigr)\,dt\Bigr)\\
        &= \gl E_v\Bigl(\int_0^\infty e^{-\gl(s+\gg L^w_s)}
            \bigl(f(W_s)-f(v)\bigr)\,(ds + \gg dL^w_s)\Bigr)\\
        &= \gl E_v\Bigl(\int_0^\infty e^{-\gl(s+\gg L^w_s)}
            \bigl(f(W_s)-f(v)\bigr)\,ds \Bigr).
\end{align*}
In the last equality we used the fact that $L^w$ only grows when $W$ is at the vertex~$v$.
By construction of the Walsh process $W$ we have
\begin{align*}
    E_v\Bigl(&e^{-\gl\gg L^w_s}\bigl(f(W_s)-f(v)\bigr)\Bigr)\\
        &= \sum_{k=1}^n w_k\,E_0\Bigl(e^{-\gl\gg L^B_s}
                \bigl(f_k(|B_s|) - f_k(0)\bigr) \Bigr)\\
        &= 2\sum_{k=1}^n w_k \int_0^\infty \int_0^\infty e^{-\gl\gg y}
                \bigl(f_k(x) - f_k(0)\bigr)\,\frac{x+y}{\sqrt{2\pi s^3}}\,
                    e^{-(x+y)^2/2s}\,dx\,dy,
\end{align*}
where we used lemma~\ref{lem1.1}. We insert the last expression above, and
use formula~\eqref{LT}. This gives
\begin{align*}
    \frac{1}{2}\,\bigl(\RsW_\gl f\bigr)''(v)
        &= 2\gl \sum_{k=1}^n w_k\,\frac{1}{\sgl + \gg\gl}\,
            \int_0^\infty e^{-\sgl x}\bigl(f_k(x)-f_k(0)\bigr)\,dx\\
        &= \frac{1}{\sgl+\gg\gl}\,\bigl(2\gl\,(e^w_\gl,f) - \sgl f(v)\bigr). \qedhere
\end{align*}
\end{proof}

From the identity $A^s \RsW_\gl = \gl\,\RsW_\gl -\text{id}$ and some simple
algebra we get the

\begin{corollary}   \label{cor4.6}
Let $\gl>0$, and $f\in\Co$. Then
\begin{equation}    \label{eq4.11}
    \RsW_\gl f(v) = \frac{1}{\sgl+\gg\gl}\,\bigl(2\,(e^w_\gl,f)+\gg f(v)\bigr)
\end{equation}
holds.
\end{corollary}

Since formula~\eqref{eq1.24} in corollary~\ref{cor1.5} is valid for the resolvent of
every Brownian motion on $\cG$, we may use that formula for $\RsW_\gl f$, sum it
against the weights $w_k$, $\kn$, and insert the right hand side of
equation~\eqref{eq4.11}. This results in
\begin{equation*}
    \sum_{k=1}^n w_k \bigl(\RsW_\gl f\bigr)'(v_k)
        = \gg\,\frac{1}{\sgl+\gg\gl}\,\bigl(2\gl\,(e^w_\gl,f)-\sgl f(v)\bigr),
\end{equation*}
and a comparison with formula~\eqref{eq4.8} shows that equation~\eqref{eq4.7} holds
true for $f$ replaced by $\RsW_\gl f$ for arbitrary $f\in\Co$. As promised we thus
have another proof of theorem~\ref{thm4.2}.

With the help of the first passage time formula we can now provide explicit
expressions for the resolvent $\RsW$, its kernel $\rsW$ and the transition
kernel $\psW$ of $\Ws$. Inserting the right hand side of equation~\eqref{eq4.11}
into the first passage time formula~\eqref{eq1.23}, we immediately obtain for
$f\in\Co$, $\gl>0$,
\begin{equation}    \label{eq4.12}
    \RsW_\gl f(\xi)
        = R^D_\gl f(\xi) + \frac{1}{\sgl +\gg\gl}\,e_\gl(\xi)
            \bigl(2\,(e^w_\gl,f) + \gg f(v)\bigr),\quad \xi\in\cG,
\end{equation}
where $R^D$ is the Dirichlet resolvent~\eqref{eq1.21}. Using formula~\eqref{eq1.22}
for the kernel of $R^D$ together with~\eqref{eq1.23}, and \eqref{eq1.24}, we get the
following result.

\begin{corollary}   \label{cor4.7}
For $\xi$, $\eta\in\cG$, $\gl>0$, the resolvent kernel $\rsW_\gl$,  of the Walsh process
with sticky vertex is given  by
\begin{equation}    \label{eq4.13}
\begin{split}
    \rsW_\gl(\xi,d\eta)
        = r^D_\gl(\xi,\eta)\,d\eta
            +\sum_{k,m=1}^n e_{\gl,k}(\xi)\,\,&2w_m\,\frac{1}{\sgl+\gg\gl}\,e_{\gl,m}(\eta)\,d\eta\\
                &\quad + \frac{\gg}{\sgl +\gg\gl}\,e_\gl(\xi)\,\epsilon_v(d\eta),
\end{split}
\end{equation}
with $r^D_\gl$ defined in~\eqref{eq1.22}, and $\epsilon_v$ denotes the Dirac measure
in $v$. Alternatively, $\rsW_\gl$ is given by
\begin{subequations}    \label{eq4.14}
\begin{equation}    \label{eq4.14a}
\begin{split}
    \rsW_\gl(\xi,d\eta)
        = r_\gl(\xi,\eta)\,d\eta
            +\sum_{k,m=1}^n e_{\gl,k}(\xi)\,&\SsW_{km}(\gl)\,\frac{1}{\sgl}\,e_{\gl,m}(\eta)\,d\eta\\
                &\quad + \frac{\gg}{\sgl +\gg\gl}\,e_\gl(\xi)\,\epsilon_v(d\eta),
\end{split}
\end{equation}
where $r_\gl$ is defined in equation~\eqref{eq1.22a}, and
\begin{equation} \label{eq4.14b}
    \SsW_{km}(\gl)
        = 2\,\frac{\sgl}{\sgl+\gg\gl}\,w_m - \gd_{km}.
\end{equation}
\end{subequations}
\end{corollary}

\begin{remark}
When all $w_m$, $m=1$, \dots, $n$, are equal to $1/n$, the matrix $S^s(\gl)$ takes the
form
\begin{equation*}
	S^s(\gl) = -\1 + \frac{2\sgl}{\sgl + \gg \gl}\,P_n
\end{equation*}
which reduces to~\eqref{eq2.6a} when $\gg=0$. $S^s(\gl)$ is unitary for all $\gl<0$.
Also the $S^s(\gl)$ for different $\gl$ all commute. As a consequence $S^s(\gl)$
has the interpretation of a quantum scattering matrix in the sense of~\cite{KoSc99}.
More precisely, $S^s(\gl)$ stems from the Schr\"odinger operator $-\gD^s$, where
$\gD^s$ is a self-adjoint Laplace operator on $L^2(\cG)$ with boundary conditions
of the form~\eqref{bc} with the choice
\begin{equation}		\label{AB}
\begin{split}
	A &= - \frac{1}{2}\,\bigl(S^s(\gl_0)+\1\bigr),\\
	B &= - \frac{1}{2\sqrt{\mathstrut2\gl_0}}\,\bigl(S^s(\gl_0)-\1\bigr),
\end{split}
\end{equation}
for any $\gl_0$ for which $\sqrt{2\gl_0}+\gg\gl_0\ne 0$. We emphasize that
the Schr\"odinger operator $-\gD^s$ and the generator $A^s$ of the Walsh process
are quite different: Not only do they act on different Banach spaces,
but also the functions in the intersection of their domains satisfy different
boundary conditions at the vertex $v$. As matter of fact, the integral kernel
of the resolvent  $(-\gD^s+2\gl)^{-1}$ of the Schr\"odinger operator $-\gD^s$ is
given by, see Lemma~4.2 in~\cite{KoSc06},
\begin{equation*}
 \frac{1}{2}\Bigl(r_\gl(\xi,\eta)
            +\sum_{k,m=1}^n e_{\gl,k}(\xi)\,
				\SsW_{km}(\gl)\,\frac{1}{\sgl}\,e_{\gl,m}(\eta)\Bigl),
\end{equation*}
that is --- up to a factor $2$ --- by the right hand side of~\eqref{eq4.14a}
\emph{without} the last term.

In more detail and with the definition~\eqref{SE}
\begin{equation*}
	S_{A,B}\bigl(E=-2\gl\bigr) = S^s(\gl)
\end{equation*}
holds for all $\gl>0$. As a function of $k$ ($k^2=E$), $S^s$ is meromorphic in the complex $k$--plane with
a pole on the positive imaginary axis at $k^b=2i/\gg$. This corresponds to a negative
eigenvalue $E^b=-4/\gg^2$ of $-\gD^s$. The corresponding (normalized) eigenfunction
$\psi^b$ --- physically speaking a \emph{bound state} --- is given as
\begin{equation*}
	\psi^b(\xi) = \frac{1}{2}\,\sqrt{\frac{\gg}{n}}\,e^{-2 d(v,\xi)/\gg},\qquad \xi\in\cG.
\end{equation*}
So quantum mechanically the vertex $v$ acts like an attractive potential. We view this
as a quantum analogue of the stickyness of the vertex $v$.

This analogy can be elaborated a bit further by inspecting the associated quantum mechanical
time delay matrix (see, e.g., \cite{AmJa77, BrFr02, JaMi72, Ma76, Ma81, Na80, Wi55})
\begin{equation*}
	T(k) = \frac{1}{2ik}\,S(k)^{-1}\frac{\p}{\p k} S(k)
\end{equation*}
which in the present context gives
\begin{equation*}
	T(k) = \frac{-2\gg}{k(4+k^2\gg^2)}\,P_n.
\end{equation*}
So $T(k)$ has zero as an $(n-1)$--fold eigenvalue plus the non-degenerate
eigenvalue
\begin{equation*}
	\frac{-2\gg}{k(4+k^2\gg^2)},
\end{equation*}
which for $\gg>0$ is the signal for a strict quantum delay. Observe that for
$k\to +\infty$, that is for large energies, the time delay experienced by the
quantum particle tends to zero, while for $k\to 0$, i.e., for low energies,
the delay becomes arbitrarily large. From the physical point of view, both
effects are clearly to be expected. For comparison and in contrast to the
present stochastic context, in quantum mechanics $\gg<0$ is also allowed for a
meaningful Schr\"odinger operator and an associated scattering matrix.
%\marginpar{\color{red} why do both resolvents satisfy the Hilbert identity???}
\end{remark}

Define for $x\ge 0$, $\gg$, $t>0$,
\begin{equation}    \label{eq4.15}
    g_{0,\gg}(t,x) = \frac{1}{\gg}\,\exp\Bigl(\frac{2x}{\gg} + \frac{2t}{\gg^2}\Bigr)\,
                    \erfc\Bigl(\frac{x}{\sqrt{2t}}+\frac{\sqrt{2t}}{\gg}\Bigr).
\end{equation}
It is not hard to check that
\begin{equation}	\label{lim_g}
    \lim_{\gg\downarrow 0} g_{0,\gg}(t,x) = g(t,x) = \frac{1}{\sqrt{2\pi t}}\,e^{-x^2/2t}.
\end{equation}
Moreover, from~\cite[eq.~5.6.16]{ErMa54a} (cf.\ also appendix~\ref{0_app_LTHK}
in~\cite{BMMG0}) the Laplace transform is
\begin{equation}    \label{eq4.16}
    \cL g_{0,\gg}(\,\cdot\,,x)(\gl) = \frac{1}{\sgl + \gg\gl}\,e^{-\sgl x},\qquad x\ge 0.
\end{equation}
Observe that in agreement with~\eqref{lim_g}
\begin{equation*}
	\cL g(\,\cdot\,,x)(\gl) = \frac{1}{\sgl}\,e^{-\sgl x}
\end{equation*}
holds. Now we can readily compute the inverse Laplace transform of formulae~\eqref{eq4.13},
\eqref{eq4.14}, and obtain the following result.

\begin{corollary}   \label{cor4.8}
For $t>0$, $\xi$, $\eta\in\cG$ the transition kernel of the Walsh process with sticky
vertex is given by
\begin{equation}    \label{eq4.17}
\begin{split}
    \psW(t,\xi,\eta) &= p^D(t,\xi,\eta)\,d\eta\\[1ex]
                &\qquad + \sum_{k,m=1}^n \onel(\xi)\,2w_m
                    \,g_{0,\gg}\bigl(t,d_v(\xi,\eta)\bigr)\onel[m](\eta)\,d\eta\\[1ex]
                &\qquad + \gg\,g_{0,\gg}\bigl(t, d(\xi,v)\bigr)\,\epsilon_v(d\eta)
\end{split}
\end{equation}
where $p^D$ is defined in equation~\eqref{eq1.20}, or alternatively by
\begin{equation}    \label{eq4.18}
\begin{split}
    \psW(t,\xi,d\eta)
        &= p(t,\xi,\eta)\,d\eta\\
        &\qquad +\sum_{k,m=1}^n \onel(\xi)\,\Bigl(
                        2w_m\,g_{0,\gg}\bigl(t,d_v(\xi,\eta)\bigr)\,d\eta\\[-1ex]
        &\hspace{9.5em}   -\gd_{km}\,g\bigl(t,d_v(\xi,\eta)\bigr)\,d\eta\Bigr)\, \onel[m](\eta)\\[1ex]
        &\qquad   +\gg\,g_{0,\gg}\bigl(t, d(\xi,v)\bigr)\,\epsilon_v(d\eta),
\end{split}
\end{equation}
and $p(t,a,b)$ is given in formula~\eqref{eq1.16}.
\end{corollary}

We close this section with some remarks concerning the local time of $\Ws$ at the
vertex $v$, which also serve to prepare the construction of the most general
Brownian motion on the single vertex graph $\cG$ in the next section.

Let us define
\begin{equation}    \label{sLT}
    L^s_t = L^w_{\tau(t)},\qquad t\ge 0,
\end{equation}
where --- as before --- $L^w$ denotes the local time of the Walsh process at the
vertex, having (cf.\ section~\ref{sect3}) the same normalization as the local time
of a standard one-dimensional Brownian motion (cf.~\eqref{eq1.7a}). By construction,
$L^s$ is pathwise continuous and non-decreasing. It is adapted to $\cF^s$, and a
straightforward calculation based on the additivity of $L^w$ and
formula~\eqref{eq4.2} shows the (pathwise) additivity property
\begin{equation}    \label{add_sLT}
    L^s_{s+t} = L^s_t + L^s_s \comp \theta^s_t,\qquad s,\,t\ge 0.
\end{equation}
Thus $L^s$ is a PCHAF of $(W^s,\cF^s)$. Furthermore, $t\ge 0$ is a point of increase
for $L^s$ if only if $\tau(t)$ is a point of increase for $L^w$, which only is the
case if $W_{\tau(t)}$ is at the vertex, i.e., if $W^s_t$ is at the vertex. Thus, it
follows that $L^s$ is a local time at the vertex for $W^s$. In order to completely
identify it, it remains to compute its normalization, and it is not very hard to
compute its $\ga$--potential (the interested reader can find the details for the
case $\cG=\R_+$ in~\cite{BMMG0}):
\begin{equation}    \label{norm_sLT}
    E_\xi\Bigl(\int_0^\infty e^{-\ga t}\,dL^s_t\Bigr)
        = \frac{1}{\sqrt{2\ga}+\gg\ga}\,e^{-\sqrt{2\ga}\,d(\xi,v)},\qquad \ga>0,\,\xi\in\cG.
\end{equation}

%   HERE, Oct. 23, 17.30

\section{The General Brownian Motion on a Single Vertex Graph}  \label{sect5}
Finally, in this subsection we construct a Brownian motion $\Wg$ by killing the Walsh
process with sticky vertex of section~\ref{sect4} in a similar way as in the
construction of the elastic Walsh process (cf.~section~\ref{sect3}). $\Wg$ realizes
the boundary condition~\eqref{eq1.1} in its most general form.

Consider the sticky Walsh process $\Ws$ with stickiness parameter $\gg>0$, right
continuous and complete filtration $\cF^s$, and local time $L^s$ at the vertex. We
argued in section~\ref{sect4} that $L^s$ is a PCHAF for $(W^s,\cF^s)$, and therefore
we can apply the method of killing described in section~\ref{ssect1.5}: We bring in
the additional probability space $(\R_+,\cB(\R_+),P_\gb)$ where $P_\gb$ is the
exponential law of rate $\gb>0$, and the canonical coordinate random variable $S$.
Then we take the family of product spaces $\bigl(\hgO,\hcA,(\hP_\xi,\,\xi\in\cG)\bigr)$
of $\bigl(\gO,\cA,(P_\xi,\,\xi\in\cG)\bigr)$ and $(\R_+,\cB(\R_+),P_\gb)$. Define
the random time
\begin{equation}    \label{eq5.1}
    \zeta_{\gb,\gg} = \inf\bigl\{t\ge 0,\,L^s_{t}>S\bigr\}.
\end{equation}
Then by the arguments given in section~\ref{ssect1.5}, the stochastic process $\Wg$
defined by $\Wg_t=\Ws_t$ for $t\in[0,\zeta_{\gb,\gg})$, and $\Wg_t=\gD$ for $t\ge
\zeta_{\gb,\gg}$, is again a Brownian motion on $\cG$ in the sense of
definition~\Iref{def_3_1}.

Denote by $K^s$ the right continuous pseudo-inverse of $L^s$. Since $L^s$
is continuous (cf.\ equation~\eqref{sLT}), we get $L^s_{K^s_r}=r$ for all
$r\in\R_+$. Recall that the right continuous pseudo-inverse of the local
time $L^w$ of the Walsh process was denoted by $K^w$. Then we have the following

\begin{lemma}	\label{lem5.1a}
For all $\gg\ge 0$, the following relation holds true:
\begin{equation}	\label{eq5.1a}
	K^s_r = K^w_r + \gg r,\qquad r\in\R_+.
\end{equation}
\end{lemma}

\begin{proof}
For $\gg$, $r\in\R_+$ define the random subset
\begin{equation*}
    J_\gg(r) = \{t\ge 0,\, L^s_t>r\}
\end{equation*}
of $\R_+$. Since $L^s$ is pathwise increasing, this set is a random interval
with endpoints $K^s_r$ and $+\infty$. The relation $L^s_{K^s_r}=r$
implies that
\begin{equation*}
	J_\gg(r) = (K^s_r,+\infty).
\end{equation*}
In particular, we have $J_0 = (K^w_r,+\infty)$. Now
\begin{equation*}
    t\in J_\gg(r) \Leftrightarrow L^s_t = L^w_{\tau(t)} > r \Leftrightarrow \tau(t)\in J_0(r).
\end{equation*}
In other words, $J_\gg(r) = \tau^{-1}\bigl(J_0(r))$, and therefore
$K^s_r=\tau^{-1}(K^w_r)$ holds. From the definition of $\tau^{-1}$ (see equation~\eqref{eq4.1}),
and the relation  $L^s_{K^s_r}=r$ we obtain formula~\eqref{eq5.1a}
\end{proof}

In the proof of lemma~\ref{lem3.3} the Laplace transform of the density of
$K^w_r$, $r\ge 0$, under $P_v$ has been determined as $\gl \mapsto
\exp(-\sgl r)$. Hence we have
\begin{equation*}
	P_v(K^w_r\in dl) = \frac{r}{\sqrt{2\pi l^3}}\,e^{-r^2/2l}\,dl,\qquad l\ge 0.
\end{equation*}
As a consequence we find the

\begin{corollary}   \label{cor5.1b}
For $r\ge 0$, $K^s_r$ has the density
\begin{equation}	\label{eq5.1b}
	P_v(K^s_r\in dl)
		= \frac{r}{\sqrt{2\pi (l-\gg r)^3}}\,e^{-r^2/2(l-\gg r)}\,dl,\qquad l\ge \gg r.
\end{equation}
Furthermore, the Laplace transform of the density of $K^s_r$ under $P_v$ is given
by
\begin{equation}	\label{eq5.1c}
	E_v\bigl(e^{-\gl K^s_r}\bigr)
		= e^{-(\sgl + \gl \gg)r},\qquad \gl >0.
\end{equation}
\end{corollary}

\begin{remark}	\label{rem5.1c}
One can use lemma~\ref{0_a_LTHK_lem1} in~\cite{BMMG0} to check that the right hand side
of equation~\eqref{eq5.1b} is indeed the inverse Laplace transform of the right hand side
of formula~\eqref{eq5.1c}.
\end{remark}

Observe that $\zeta_{\gb,\gg} = K^s_S$ and $\zeta_{\gb,0}=K^w_S$. Thus we obtain the

\begin{corollary}   \label{lem5.1}
For all $\gb>0$, $\gg\ge0$, the following equation holds true
\begin{equation}    \label{eq5.2}
    \zeta_{\gb,\gg} = \zeta_{\gb,0}+\gg S.
\end{equation}
\end{corollary}

As before, $\hE_\xi$ denotes the expectation with respect to $\hP_\xi$, $\xi\in\cG$.

\begin{corollary}   \label{cor5.2}
For all $\gb>0$, $\gg\ge 0$, $\gl>0$, the following formula holds true
\begin{subequations}    \label{eq5.3}
\begin{equation}    \label{eq5.3a}
    \hE_v\bigl(e^{-\gl \zeta_{\gb,\gg}}\bigr) = \gb \rho(\gl),
\end{equation}
with
\begin{equation}    \label{eq5.3b}
    \rho(\gl) = \frac{1}{\gb+\sgl+\gg\gl}.
\end{equation}
\end{subequations}
\end{corollary}

\begin{proof}
With corollary~\ref{cor5.1b} and $\zeta_{\gb,\gg} = K^s_S$ we obtain
\begin{align*}
    \hE_v\bigl(e^{-\gl \zeta_{\gb,\gg}}\bigr)
        &= \gb \int_0^\infty E_v\bigl(e^{-\gl K^s_r}\bigr)\,e^{-\gb r}\,dr\\
	   &= \frac{\gb}{\gb + \sgl + \gg \gl}.	\qedhere
\end{align*}
\end{proof}

Denote by $\Rg$ the resolvent of $\Wg$. With lemma~\ref{lem1.6} we immediately find
the

\begin{corollary}   \label{cor5.3}
For all $f\in\Co$, $\gl>0$, $\xi\in\cG$ the following formula holds true:
\begin{equation}    \label{eq5.4}
    \Rg_\gl f(\xi)
        = \RsW_\gl f(\xi) - \gb\rho(\gl)\,e_\gl(\xi)\,\RsW_\gl f(v).
\end{equation}
\end{corollary}

Now it is easy to verify that for appropriately chosen parameters $\gb$, $\gg$,
$w_k$, $\kn$, the Brownian motion $W_g$ realizes the boundary
condition~\eqref{eq1.1b}.

\begin{theorem} \label{thm5.4}
Consider the boundary condition~\eqref{eq1.1}, and assume that $b$ is not the null
vector. Set $r=a+c\in(0,1)$, and
\begin{equation}    \label{eq5.5a}
    w_k = \frac{b_k}{1-r},\,\kn,\quad \gb = \frac{a}{1-r},\quad \gg = \frac{c}{1-r}.
\end{equation}
Let $\Wg$ be the Brownian motion as constructed above with these parameters.
Then the generator $A^g$ of $\Wg$ is $1/2$ times the Laplacean on $\cG$ with domain
$\cD(A^g)$ consisting of those $f\in\Cii$ which satisfy condition~\eqref{eq1.1b}.
\end{theorem}

\begin{proof}
As in the previous cases it is readily seen that the definition~\eqref{eq5.5a} of
the parameters $\gg$, $\gb$, $w_k$, $\kn$, is consistent with the conditions used in
the above construction of $\Wg$.

Let $A^g$ be the generator of $\Wg$ with domain $\cD(A^g)$. Since $\Wg$ is a Brownian
motion on $\cG$ in the sense of definition~\Iref{def_3_1}, it follows from
lemma~\Iref{lem_5_2} that $\cD(A^g)\subset \Cii$, and that for all $f\in\cD(A^g)$,
$A^g f(\xi) = 1/2\,f''(\xi)$, $\xi\in\cG$. Let $h\in\Co$, $\gl>0$. Then $\Rg_\gl
h\in\cD(A^g)$, and therefore we may compute with equation~\eqref{eq5.4} as follows
\begin{align*}
    \frac{\gg}{2}\,\bigl(\Rg_\gl h\bigr)''(v)
        &= \frac{\gg}{2}\,\bigl(\RsW_\gl h\bigr)''(v)
            - \gb\,\rho(\gl)\,2\gl \bigl(\RsW_\gl h\bigr)(v)\\
        &= \sum_{k=1}^n w_k\,\bigl(\RsW_\gl h\bigr)'(v_k)
            - \gb\,\rho(\gl)\,\gg\gl \bigl(\RsW_\gl h\bigr)(v),
\end{align*}
where we used the fact that, since $\RsW_\gl h$ is in the domain of the generator $A^s$ of
$\Ws$, it satisfies the boundary condition~\eqref{eq4.7}. We rewrite this equation
in the following way:
\begin{equation}    \label{eq5.5}
\begin{split}
    \frac{\gg}{2}\,\bigl(\Rg_\gl h\bigr)''(v)
        = \sum_{k=1}^n w_k\,\bigl(\RsW_\gl h\bigr)'(v_k)
                &+\gb\sgl\,\rho(\gl) \bigl(\RsW_\gl h\bigr)(v)\\
                &-\gb\,\rho(\gl)\bigl(\sgl+\gg\gl\bigr)\,\bigl(\RsW_\gl h\Bigr)(v).
\end{split}
\end{equation}
Now we differentiate equation~\eqref{eq5.4} at $\xi\in l_k$, $\kn$, let $\xi$ tend to
$v$ along any edge $l_k$, and sum the resulting equation against the weights $w_k$,
$\kn$. Then we get the following formula
\begin{equation}    \label{eq5.6}
    \sum_{k=1}^n w_k\,\bigl(\Rg_\gl h\bigr)'(v_k)
        = \sum_{k=1}^n w_k\,\bigl(\RsW_\gl h\bigr)'(v_k)
            + \gb\sgl\,\rho(\gl)\bigl(\RsW_\gl h\bigr)(v),
\end{equation}
where we used $\sum_k w_k=1$. On the other hand, for $\xi=v$, equation~\eqref{eq5.4}
gives
\begin{equation}    \label{eq5.7}
    \bigl(\Rg_\gl h\bigr)(v) = \rho(\gl)\bigl(\sgl +\gg\gl\bigr)\bigl(\RsW_\gl h\bigr)(v).
\end{equation}
A comparison of equations~\eqref{eq5.6}, \eqref{eq5.7} with \eqref{eq5.5} shows
that we have proved the following formula
\begin{equation}    \label{eq5.8}
    \frac{\gg}{2}\,\bigl(\Rg_\gl h\bigr)''(v)
        = \sum_{k=1}^n w_k\,\bigl(\Rg_\gl h\bigr)'(v_k) - \gb \bigl(\Rg_\gl h\bigr)(v).
\end{equation}
With the values~\eqref{eq5.5a} for $\gb$, $\gg$, and $w_k$, $\kn$, it is obvious that
$f=\Rg_\gl h$ satisfies equation~\eqref{eq1.1b}. Since $\Rg_\gl$ is surjective from
$\Co$ onto the domain of the generator $A^g$ of $\Wg$, the proof of the theorem is finished.
\end{proof}

Let $\gl>0$, $f\in\Co$. Insertion of the right hand side of formula~\eqref{eq4.12}
for $\RsW_\gl$ into equation~\eqref{eq5.4} gives us after some simple algebra the
following expression for $\Rg_\gl f$:
\begin{equation}    \label{eq5.9}
    \Rg_\gl f(\xi)
        = R^D_\gl f(\xi) + \rho(\gl)\,e_\gl(\xi)\,\bigl(2(e_\gl^w,f) + \gg f(v)\bigr),
            \qquad \xi\in\cG,
\end{equation}
where $R^D$ is the Dirichlet resolvent, $e_\gl$ is defined in equation~\eqref{eq1.14},
$e^w_\gl$ in equation~\eqref{eq4.9}, and $\rho(\gl)$ is as in formula~\eqref{eq5.3b}.
From equation~\eqref{eq5.9} we can read off the following result:

\begin{corollary}   \label{cor5.5}
For $\xi$, $\eta\in\cG$, $\gl>0$, the resolvent kernel $\rg_\gl$ of the general Brownian
motion $\Wg$ on $\cG$ is given  by
\begin{equation}    \label{eq5.10}
\begin{split}
    \rg_\gl(\xi,d\eta)
        = r^D_\gl(\xi,\eta)\,d\eta
            +\sum_{k,m=1}^n e_{\gl,k}&(\xi)\,2w_m\,\rho(\gl)\,e_{\gl,m}(\eta)\,d\eta\\
                & + \gg\,\rho(\gl)\,e_\gl(\xi)\,\epsilon_v(d\eta),
\end{split}
\end{equation}
with $r^D_\gl$ as in formula~\eqref{eq1.22}, and $\rho$ is defined in
equation~\eqref{eq5.3b}. Alternatively, $\rg_\gl$ can be written in the following form
\begin{subequations}    \label{eq5.11}
\begin{equation}    \label{eq5.11a}
\begin{split}
    \rg_\gl(\xi,d\eta)
        = r_\gl(\xi,\eta)\,d\eta
            +\sum_{k,m=1}^n e_{\gl,k}&(\xi)\,\Sg_{km}(\gl)
                \,\frac{1}{\sgl}\,e_{\gl,m}(\eta)\,d\eta\\
                & + \gg\,\rho(\gl)\,e_\gl(\xi)\,\epsilon_v(d\eta),
\end{split}
\end{equation}
where $r_\gl$ is defined in equation~\eqref{eq1.22a}, and
\begin{equation} \label{eq5.11b}
    \Sg_{km}(\gl)
        = 2\,\sgl\,\rho(\gl)\,w_m - \gd_{km}.
\end{equation}
\end{subequations}
\end{corollary}

In order to invert the Laplace transforms in equations~\eqref{eq5.10},
\eqref{eq5.11}, we define for $\gb$, $\gg>0$, the following function $g_{\gb,\gg}$
on $(0,+\infty)\times \R_+$:
\begin{equation}    \label{eq5.12}
    g_{\gb,\gg}(t,x)
        = \frac{1}{\gg^2}\,\frac{1}{\sqrt{2\pi}}\,\int_0^t \frac{s+\gg x}{(t-s)^{3/2}}\,
            \exp\Bigl(-\frac{(s+\gg x)^2}{2\gg^2(t-s)}\Bigr)\,e^{-\gb s/\gg}\,ds,
\end{equation}
with $(t,x)\in(0,+\infty)\times\R_+$. The heat kernel $g_{\gb,\gg}$ is discussed in
more detail in appendix~\ref{0_app_LTHK} of~\cite{BMMG0}. In particular, it is
outlined there that the limits of $g_{\gb,\gg}$ as $\gb\downarrow 0$, and
$\gg\downarrow 0$, yield the kernels $g_{\gb,0}$ (equation~\eqref{eq3.10}) and
$g_{0,\gg}$ (equation~\eqref{eq4.15}), respectively. Moreover, it is proved there
that the Laplace transform of $g_{\gb,\gg}(\cdot,x)$, $x\ge 0$, is given by
\begin{equation}    \label{eq5.13}
    \rho(\gl)\,e^{-\sgl x},\qquad \gl>0,
\end{equation}
where $\rho$ is defined in~\eqref{eq5.3b}. Hence we get the

\begin{corollary}   \label{cor5.6}
For $\xi$, $\eta\in\cG$, $t>0$, the transition kernel of the general Brownian
motion $\Wg$ on $\cG$ is given  by
\begin{equation}    \label{eq5.14}
\begin{split}
    \pg(t,\xi,d\eta)
        &= p^D(t,\xi,\eta)\,d\eta\\
        &\hspace{2em} +\sum_{k,m=1}^n \onel(\xi)\,2w_m\,g_{\gb,\gg}\bigl(t,d_v(\xi,\eta)\bigr)\,
                                \onel[m](\eta)\,d\eta\\
        &\hspace{2em} + \gg\,g_{\gb,\gg}\bigl(t,d(\xi,v)\bigr)\,\gep_v(d\eta),
\end{split}
\end{equation}
which alternatively can be written as
\begin{equation}    \label{eq5.15}
\begin{split}
    \pg(t,\xi,d\eta)
        &= p(t,\xi,\eta)\,d\eta\\
        &\hspace{2em} +\sum_{k,m=1}^n \onel(\xi)\,\Bigl(2w_m\,g_{\gb,\gg}\bigl(t,d_v(\xi,\eta)\bigr)\\
        &\hspace{10em} -\gd_{km}\,g\bigl(t,d_v(\xi,\eta)\bigr)\Bigl)\,
                                \onel[m](\eta)\,d\eta\\
        &\hspace{2em} + \gg\,g_{\gb,\gg}\bigl(t,d(\xi,v)\bigr)\,\gep_v(d\eta).
\end{split}
\end{equation}
\end{corollary}

\providecommand{\bysame}{\leavevmode\hbox to3em{\hrulefill}\thinspace}
\providecommand{\MR}{\relax\ifhmode\unskip\space\fi MR }
\providecommand{\MRhref}[2]{%
  \href{http://www.ams.org/mathscinet-getitem?mr=#1}{#2}
}
\providecommand{\href}[2]{#2}

\end{document}